\documentclass[a4paper,11pt]{article}
\usepackage{amsmath,amssymb,amsthm,graphicx,subfigure,float,caption,epstopdf,tabularx,color, bm,amsfonts,epic}
\usepackage[top=1in, bottom=1in, left=1.25in, right=1.25in]{geometry}
\usepackage{appendix}
\usepackage{multirow}
\usepackage{diagbox}
\allowdisplaybreaks
\newtheorem{thm}{Theorem}[section]
\newtheorem{lem}{Lemma}[section]

\numberwithin{equation}{section}

\usepackage{indentfirst}
\graphicspath{{Fig}}

\numberwithin{equation}{section} \numberwithin{figure}{section}
\numberwithin{table}{section}

\newtheorem{rem}{Remark}[section]
\newtheorem{scheme}{Scheme}[section]

\def\bu{\mathbf{u}}
\def\bu{\mathbf{v}}

\def\bx{\mathbf{x}}

\def\bD{\mathbf{D}}

\def\bv{\mathbf{v}}
\def\bB{\mathbf{B}}

\newcommand{\ben}{\begin{eqnarray}}
\newcommand{\een}{\end{eqnarray}}
\newcommand{\beq}{\begin{equation}}
\newcommand{\eeq}{\end{equation}}
\newcommand{\bea}{\begin{array}}
\newcommand{\eea}{\end{array}}
\newcommand{\bef}{\begin{figure}[t]}
\newcommand{\eef}{\end{figure}}

\newcommand{\benl}{\begin{eqnarray*}}
\newcommand{\eenl}{\end{eqnarray*}}
\newcommand{\be}{\begin{equation}}
\newcommand{\ee}{\end{equation}}

\newcommand{\bse}{\begin{subequations}}
\newcommand{\ese}{\end{subequations}}

\def\bA{\mathbf{A}}
\def\cL{\mathcal{L}}

\def\cG{\mathcal{G}}

\def\bx{\mathbf{x}}
\def\cL{\mathcal{L}}
\def\cG{\mathcal{G}}

\begin{document}

\title{Arbitrarily High-order Unconditionally Energy Stable Schemes for Gradient Flow Models Using the Scalar Auxiliary Variable Approach}
\author{
Yuezheng Gong \footnote{College of Science, Nanjing University of Aeronautics and Astronautics, Nanjing 210016, China; Email: gongyuezheng@nuaa.edu.cn.}
, Jia Zhao \footnote{Department of Mathematics \& Statistics, Utah State University, Logan, UT, USA; email: jia.zhao@usu.edu.}
and Qi Wang \footnote{Department of Mathematics, University of South Carolina, Columbia, SC, 29208, USA; email: qwang@math.sc.edu.}}
\date{}
\maketitle

\begin{abstract}
In this paper, we propose a novel family of high-order numerical schemes for the gradient flow models based on the scalar auxiliary variable (SAV) approach, which is named the high-order scalar auxiliary variable (HSAV) method. The newly proposed schemes could be shown to reach arbitrarily high order in time while preserving the energy dissipation law without any restriction on the time step size (i.e., unconditionally energy stable). The HSAV strategy is rather general that it does not depend on the specific expression of the effective free energy, such that it applies to a class of thermodynamically consistent gradient flow models arriving at semi-discrete high-order energy-stable schemes. We then employ the Fourier pseudospectral method for spatial discretization. The fully discrete schemes are also shown to be unconditionally energy stable. Furthermore, we present several numerical experiments on several widely-used gradient flow models, to demonstrate the accuracy, efficiency and unconditionally energy stability of the HSAV schemes. The numerical results verify that the HSAV schemes can reach the expected order of accuracy, and it allows a much larger time step size to reach the same accuracy than the standard SAV schemes.
\end{abstract}


\section{Introduction}
The dynamics of many dissipative systems could be driven by an effective free energy that is decreasing with time, where the decreasing path is controlled by a certain dissipation mechanism \cite{OnsagerL1,OnsagerL2,ZhaoYangGongZhaoYangLiWang2018,YangLiForestWang2016}. To study such dynamics in the macroscopic level, a gradient flow model is usually used. In general, consider the state variables $\Phi(\bx,t)$ for a dissipative system on the domain $\Omega$. The evolution (kinetic) equation for $\Phi(\bx,t)$ could be formulated as \cite{SAV-1}
\beq \label{eq:Intro_gradient}
\partial_t \Phi(\bx,t) = \cG \frac{\delta F}{\delta \Phi},
\eeq
where $\cG$ is a negative semi-definite differential or integral operator which might depend on the state variables $\Phi$. Here $F$ is the effective free energy, and $\frac{\delta F}{\delta \Phi}$ is the variational derivative of $F$ with respect to the state variable $\Phi$, called the chemical potential. For instance, if $F=F(\Phi, \nabla \Phi)$, the chemical potential would be $\frac{\delta F}{\delta \Phi} = \frac{\partial F}{\partial \Phi} - \nabla \cdot \Big( \frac{\partial F}{\partial \nabla \Phi} \Big)$. In this paper, we consider periodic boundary conditions for simplicity of notations. 

The system \eqref{eq:Intro_gradient} is a general gradient flow model, which has an intrinsic energy dissipation law. Actually, if we take inner product of \eqref{eq:Intro_gradient} with the chemical potential $\frac{\delta F}{\delta \Phi}$, we have the energy dissipation law
\beq
\frac{d F}{d t} = \int_\Omega \Big( \frac{\delta F}{\delta \Phi} \Big)^T \cG \frac{\delta F}{\delta \Phi} d\bx \leq 0,
\eeq
thanks to the negative semi-definite property of $\cG$.

Note that the gradient flow model could be specified once the triple $(\Phi,\cG,F)$ is given. It turns out many dissipative PDE models could be classified as a special case of the general gradient flow model in \eqref{eq:Intro_gradient}. For instance, if we specify the state variable as $\phi$, the mobility as $\cG = -M$ (with $M$ a positive constant) and the free energy as $F=\int_\Omega \big[\frac{1}{2}|\nabla \phi|^2 + \frac{1}{4\varepsilon^2}(\phi^2-1)^2\big]d \bx$ (with $\varepsilon$ a free parameter), the general gradient flow model \eqref{eq:Intro_gradient} reduces to the well-known Allen-Cahn equation \cite{AC}
\beq
\partial_t \phi = -M \Big( -\Delta \phi + \frac{1}{\varepsilon^2}(\phi^3-\phi)\Big).
\eeq
If we use the same free energy and specify the mobility as $\cG= M\Delta$, we end up with the well-known Cahn-Hilliard equation \cite{CahnH1958}
\beq
\partial_t \phi = M \Delta \Big( -\Delta \phi + \frac{1}{\varepsilon^2}(\phi^3-\phi)\Big).
\eeq
Moreover, there are many more examples which could be cast as special cases of \eqref{eq:Intro_gradient}, including molecular beam epitaxy (MBE) growth models  \cite{WangWangWiseDCDS2010,ChenZhaoYang2018}, phase field crystal models \cite{Gomez-crystal}, dendritic crystal growth models \cite{WangPhysicaD1993}, multiphase models \cite{Boy2006,KimJCICP2012}.  For a detailed discussion,  readers can refer to  \cite{ZhaoYangGongZhaoYangLiWang2018} and the references therein.

Along with the broad applications of gradient flow models, many accurate, efficient and stable numerical schemes are developed for longtime dynamic simulations of the dissipative systems. As the dynamics follow a specified trace of dissipating the effective free energy, one essential indication of stable numerical schemes is to preserve the energy dissipation law at the discrete level. A numerical scheme that possesses such property is known as energy stable. If such numerical stability does not have any restrictions on the time step size, it is then usually named unconditionally energy stable. In practice, energy dissipation preserving schemes are always desirable as they mimic the physical structures of the original problem and thus perform excellent numerical stability even with large marching time step size. Due to its practical significance, a large number of innovative work have been developed, please see \cite{Eyre1998Unconditionally,WangWangWiseDCDS2010,WangWiseSINUMA2011,Shen10_1,RomeroNME2009,GuillienTierraJCP2013,BadalassiCenicerosBanerjeeJCP2003,LiKimCNSNS2017,LaiEAJAM2013,JuLiQiaoZhangMC2017,ChenCondeWangWangWiseJSC2012,GuoLinLowengrubWise2017,YangZhaoHeJCAM2018,YangZhaoCICP2018} and the references therein.  
However, most of these existing schemes have strict restrictions on $\cG$ and $F$, i.e., they only work for particular gradient flow models. Recently, Yang et al. \cite{YangZhaoWangShenM3AS2017, ZhaoYangGongZhaoYangLiWang2018, GongZhaoWangSISC2, GongZhaoYangWangSISC,HanBrylevYangTanJSC2017} propose an energy quadratization (EQ) approach to bypass the restrictions and obtain linear energy stable schemes, which can be applied for almost all gradient flow problems. Shen et al. \cite{SAV-1, SAV-2} further extend the EQ idea to develop the scalar auxiliary variable (SAV) approach, where the resulting linear schemes can be solved quickly by the fast Fourier transform (FFT). The primary idea of both EQ and SAV is to introduce some auxiliary variables to reformulate a gradient flow model into an equivalent form, such that the effective free energy is a quadratic functional in the reformulated equivalent system. Then linear and unconditionally energy stable schemes could be constructed easily for the reformulated system, which in turn solves the original gradient flow model. Due to the generality of the EQ and SAV approaches, they have been applied for many existing gradient flow models \cite{YangZhaoHeJCAM2018, YangZhaoCICP2018, YangZhaoWangShenM3AS2017, SAV-1, SAV-2, ChenZhaoYang2018,XuYIEQ}.

However, most of the existing energy stable schemes are up to second-order accurate in time. There is little work on developing higher order energy stable schemes.  Since gradient flow models usually require longtime dynamic simulations to reach the steady state, high order accurate energy stable schemes are always desirable, which makes large marching steps practical while preserving the accuracy. Several seminal works on developing high-order energy-stable numerical schemes include \cite{RKJCP2017, GongZhaoAML2019}. In this paper, we take advantage of the SAV idea \cite{SAV-1, SAV-2} to develop arbitrarily high-order energy-stable numerical schemes for gradient flow models in two steps: firstly, we utilize the SAV technique to transform the gradient glow model into an equivalent form. The transformed equivalent system turns out to have a quadratic free energy functional along with a modified energy dissipation law; secondly, we exploit the structure-preserving Gaussian collocation and RK methods to derive high-order scalar auxiliary variable (HSAV) schemes, which are proved rigorously to preserve the discrete modified energy dissipation law. Note that the newly proposed high-order schemes overcome all the drawbacks of the convex-splitting RK scheme in \cite{RKJCP2017}. For instance, by utilizing the SAV technique, our approach does not have any restrictions on specific forms of the mobility $\cG$ and effective free energy $F$, making it applicable to all the existing gradient flow models. Moreover, by employing the Gaussian collocation method, our approach can reach arbitrarily high-order accuracy in time with optimal RK stages.

The rest of this paper is organized as follows. In Section \ref{sec:model}, we will present the general gradient flow model and its equivalent reformulation based on the SAV approach. In Section \ref{sec:TimeDis}, we derive the high-order time discretization for the reformulated system and prove its unconditional energy stability. In Section \ref{sec:spaceDis}, we use the Fourier pseudo-spectral method for spatial discretization to arrive at fully discrete schemes, which are shown to be unconditionally energy stable as well. Several numerical examples are presented in Section \ref{sec:Numer}. In the end, we give the concluding remarks.

\section{Gradient Flow Models and Their SAV Reformulation} \label{sec:model}
In this section, we start from general gradient flow models and apply the SAV approach to derive an equivalent form, which has a quadratic energy functional and an energy dissipation law for the new system. We call this approach energy quadratization reformulation (or SAV reformulation). The EQ/SAV reformulation for the gradient flow models provides an elegant platform for developing arbitrarily high-order unconditionally energy stable schemes, which is the major focus of this paper.

\subsection{General gradient flow models}
Consider the material domain $\Omega$ with enough regularity on the boundary. The $L^2$ inner product and its norm are defined as $\forall f, g \in L^2(\Omega)$, $(f,g) = \int_\Omega f g d\bx$ and $\| f \|_2 = \sqrt{(f,f)}$, respectively. For simplicity, we consider a single state variable $\phi$. The dynamics of $\phi$ is driven by an effective free energy or Lyapunov function $F$ and a negative semi-definite mobility operator $\cG$. Thus the gradient flow model is formulated as
\beq \label{eq:gradeint_flow}
\partial_t \phi = \cG \frac{\delta F}{\delta \phi}.
\eeq

The generic form of the effective free energy $F$ could be written as
\beq \label{eq:gradient_F}
\displaystyle F = \frac{1}{2}(\cL \phi, \phi) + (g, 1),
\eeq
where $\cL$ is a linear, self-adjoint operator, and $g$ is a potential functional that might depend on $\phi$ and its low order spatial derivatives. For instance, given the widely used Ginzburg-Landau free energy functional for the two phase immersible materials
\beq
F = \int_\Omega \gamma \left[\frac{\varepsilon}{2}|\nabla \phi|^2 + \frac{1}{\varepsilon} \phi^2 (1-\phi)^2\right] d\bx,
\eeq
where $\gamma$ is the surface tension, and $\varepsilon$ is the interfacial thickness, we can cast
\beq
\cL= -\gamma \varepsilon \Delta + \gamma_0 , \quad g(\phi) = \frac{\gamma}{\varepsilon}\phi^2(1-\phi)^2 - \frac{\gamma_0}{2} \phi^2,
\eeq
where $\gamma_0$ is a non-negative constant (stabilization parameter \cite{ChenZhaoYang2018}), by assuming periodic boundary conditions or other proper boundary conditions such that the boundary integral terms are canceled out.

With the specific form of effective free energy F in \eqref{eq:gradient_F}, the gradient flow model \eqref{eq:gradeint_flow} could be rewritten as
\beq \label{eq:gradient_flow}
\partial_t \phi = \cG \Big( \cL\phi + \frac{\delta g}{\delta \phi} \Big).
\eeq
Here we note that $\cL$ is self-adjoint and $\cG$ is negative semi-definite, i.e., under proper boundary conditions they satisfy
\beq
(\cL\phi,\psi) = (\phi,\cL\psi), \quad (\psi,\cG\psi)\leq 0,\quad \forall \phi,\psi \in L^2(\Omega).
\eeq
Therefore, the gradient flow system \eqref{eq:gradient_flow} satisfies the following energy dissipation law
\beq \label{eq:energy_dissipation}
\frac{dF}{dt} = \Big(\cL\phi + \frac{\delta g}{\delta \phi}, \partial_t \phi \Big) = \Big( \cL\phi + \frac{\delta g}{\delta \phi}, \cG (\cL\phi + \frac{\delta g}{\delta \phi}) \Big) \leq 0.
\eeq

\subsection{Model reformulation using the SAV approach}
For simplicity of notations, we assume $g$ only depends on $\phi$, but not its spatial derivatives. But we note the SAV approach works for a more general $g$. We first introduce a scalar auxiliary variable
\beq \label{eq:SAV_q}
q(t) = \sqrt{\big(g(\phi),1\big)+C_0},
\eeq
where $C_0$ is a positive number such that $\big(g(\phi),1\big) + C_0>0$. Then we can reformulate the original gradient flow system \eqref{eq:gradient_flow} into the following equivalent PDEs
\beq \label{eq:gradient_flow_SAV}
\begin{cases}
\partial_t \phi = \cG \Big( \cL \phi + \frac{q g'(\phi) }{\sqrt{\big(g(\phi),1\big)+C_0}} \Big), \\
\partial_t q = \left(\frac{g'(\phi)}{2 \sqrt{\big(g(\phi),1\big)+C_0}}, \,\,\, \partial_t \phi \right),
\end{cases}
\eeq
with the consistent initial condition
\beq \label{eq:gradient_flow_SAV_initial}
q|_{t=0} =\sqrt{\big(g(\phi|_{t=0}),1\big)+C_0}.
\eeq
It is obvious that the new system \eqref{eq:gradient_flow_SAV} along with the consistent initial condition \eqref{eq:gradient_flow_SAV_initial} is equivalent to the original gradient flow system \eqref{eq:gradient_flow}. So in our latter discussion, we are going to design arbitrarily high-order numerical approximations for the equivalent model \eqref{eq:gradient_flow_SAV}, which in turn solves the original gradient flow problem \eqref{eq:gradient_flow}.

In the reformulated system \eqref{eq:gradient_flow_SAV}, the modified free energy could be defined as
\beq
E = \frac{1}{2}(\cL\phi, \phi) + q^2 - C_0,
\eeq
which is equal to the free energy $F$ of the original system \eqref{eq:gradient_flow} in the continuous level, by noticing \eqref{eq:SAV_q}. As we have been emphasizing, the free energy of the equivalent system \eqref{eq:gradient_flow_SAV} is a quadratic functional with respect to the new variables.
And the new system \eqref{eq:gradient_flow_SAV} satisfies the modified energy dissipation law
\ben \label{eq:SAV_energy_dissipation}
\frac{d E}{d t} = (\cL\phi, \phi_t) + 2q q_t = \left(\cL\phi  + \frac{q g'(\phi)}{\sqrt{\big(g(\phi),1\big)+C_0}} , \phi_t \right)\nonumber \\
= \left( \cL\phi  + \frac{q g'(\phi)}{\sqrt{\big(g(\phi),1\big)+C_0}} , \cG \Big(\cL\phi  + \frac{q g'(\phi)}{\sqrt{\big(g(\phi),1\big)+C_0}} \Big)\right) \leq 0.
\een

Next we will focus on the SAV reformulated system \eqref{eq:gradient_flow_SAV}-\eqref{eq:gradient_flow_SAV_initial} to develop arbitrarily high-order unconditionally energy stable numerical approximations, which in turn solve \eqref{eq:gradient_flow}.

\section{High Order Time Discretization}\label{sec:TimeDis}
In this section, we first derive the RK method and the collocation method in time for the SAV reformulated system \eqref{eq:gradient_flow_SAV}, respectively. Then both a class of RK methods and the collocation methods with the Gaussian quadrature nodes are proved to preserve the corresponding energy dissipation law and thus unconditionally energy stable. Note that the proposed methods can reach arbitrarily high order while preserving the modified energy dissipation law.

Applying an $s$-stage RK method to solve the system  \eqref{eq:gradient_flow_SAV}, we obtain the following HSAV-RK scheme.

\begin{scheme}[$s$-stage HSAV-RK Method] \label{scheme:RK-method}
Let $b_i$, $a_{ij}$ ($i,j = 1,\cdots,s$) be real numbers and let $c_i = \sum\limits_{j=1}^s a_{ij}$.
For given $(\phi^n, q^n)$, the following intermediate values are first calculated by
\beq
\bea{l}
\Phi_i =  \phi^n +  \Delta t \sum\limits_{j=1}^s a_{ij} k_j, \\
Q_i = q^n +  \Delta t \sum\limits_{j=1}^s a_{ij} l_j, \\
k_i = \cG \left( \cL \Phi_i + \frac{Q_i g'(\Phi_i) }{\sqrt{\big(g(\Phi_i),1\big)+C_0}} \right)  \\
l_i = \left(\frac{g'(\Phi_i)}{2 \sqrt{\big(g(\Phi_i),1\big)+C_0}}, k_i \right).
\eea
\eeq
Then $(\phi^{n+1}, q^{n+1})$ is updated via
\beq
\bea{l}
\phi^{n+1} = \phi^n +  \Delta t \sum\limits_{i=1}^s b_i k_i , \\
q^{n+1} = q^n +  \Delta t \sum\limits_{i=1}^s b_i l_i.
\eea
\eeq
\end{scheme}
\noindent The RK coefficients are usually displayed by a Butcher table
\[
\begin{array}
{c|c}
\mathbf{c}  \mathbf{A} \\
\hline
  \mathbf{b}^T\\
\end{array}
=
\begin{array}
{c|ccc}
c_1  a_{11}  \cdots  a_{1s}  \\
\vdots  \vdots    \vdots \\
c_s  a_{s1}  \cdots  a_{ss} \\
\hline
 b_1   \cdots  b_s \\
\end{array},
\]
where $\mathbf{A} \in \mathbb{R}^{s,s}$, $\mathbf{b} \in \mathbb{R}^s$, and $\mathbf{c}=\mathbf{A} \mathbf{l}$ with $\mathbf{l}=(1,1,\cdots,1)^T \in \mathbb{R}^s$.

For general HSAV-RK methods, we have the following energy-stability theorem.
\begin{thm} \label{thm:High-SAV-RK}
If the coefficients of an HSAV-RK method satisfy
\beq\label{eq:RK-stable-condition}
b_i a_{i j} + b_j a_{j i} = b_i b_j,\quad b_i\geq 0,\quad \forall~i,j = 1,\cdots,s,
\eeq
then it is unconditionally energy stable, i.e., it satisfies the following energy dissipation law
\beq\label{eq:High-EL-RK}
E^{n+1}-E^n = \Delta t \sum_{i=1}^s b_i \Big(  \cL \Phi_i + \frac{ Q_i g'(\Phi_i )}{\sqrt{\big(g(\Phi_i ), 1\big)+C_0}} , \cG\Big[ \cL \Phi_i + \frac{ Q_i g'(\Phi_i )}{\sqrt{\big(g(\Phi_i ), 1\big)+C_0}} \Big] \Big)\leq 0,
\eeq
where $E^{n} = \frac{1}{2}(\cL \phi^n, \phi^n) + (q^n)^2 - C_0$.
\end{thm}

\begin{proof}
Denoting $\phi^{n+1} = \phi^n + \Delta t \sum\limits_{i=1}^s b_i k_i$ and noticing that the operator $\cL$ is linear and self-adjoint, we have
\beq\label{energy-delta-phi-1}
\frac{1}{2}(\cL \phi^{n+1} , \phi^{n+1}) - \frac{1}{2}(\cL \phi^{n} , \phi^{n}) = \Delta t\sum\limits_{i=1}^s b_i(k_i,\cL \phi^{n}) +\frac{\Delta t ^2}{2}\sum\limits_{i,j=1}^s b_i b_j (k_i,\cL k_j).
\eeq
Applying $\phi^{n} = \Phi_i - \Delta t\sum\limits_{j=1}^s a_{ij} k_j$ to the right of \eqref{energy-delta-phi-1}, we can deduce
\beq\label{energy-delta-phi-2}
\frac{1}{2}(\cL \phi^{n+1} , \phi^{n+1}) - \frac{1}{2}(\cL \phi^{n} , \phi^{n}) = \Delta t\sum\limits_{i=1}^s b_i(k_i,\cL \Phi_i),
\eeq
where $\sum\limits_{i,j=1}^s b_i a_{ij} (k_i,\cL k_j) = \sum\limits_{i,j=1}^s b_j a_{j i} (k_i,\cL k_j)$ and $b_i a_{i j} + b_j a_{j i} = b_i b_j$ were used. Similarly, we have
\beq\label{energy-delta-q}
|q^{n+1}|^2 - |q^n|^2 = 2\Delta t\sum\limits_{i=1}^s b_i l_i Q_i = \Delta t \sum\limits_{i=1}^s b_i \Big( \frac{ Q_i g'(\Phi_i )}{\sqrt{\big(g(\Phi_i ), 1\big)+C_0}} , k_i \Big).
\eeq
Adding \eqref{energy-delta-phi-2} and \eqref{energy-delta-q} leads to
\beq\label{delta-energy}
E^{n+1} - E^n = \Delta t \sum\limits_{i=1}^s b_i \Big( \cL \Phi_i + \frac{ Q_i g'(\Phi_i )}{\sqrt{\big(g(\Phi_i ), 1\big)+C_0}} , k_i \Big).
\eeq
Replacing $k_i = \cG \left( \cL \Phi_i + \frac{Q_i g'(\Phi_i) }{\sqrt{\big(g(\Phi_i),1\big)+C_0}} \right)$ to \eqref{delta-energy}, we can arrive at \eqref{eq:High-EL-RK}. This completes the proof.
\end{proof}

Applying an $s$-stage collocation method for the system \eqref{eq:gradient_flow_SAV}, we obtain the following HSAV-Collocation scheme.
\begin{scheme} [$s$-stage HSAV Collocation Method] \label{eq:Collocation-Scheme}
Let $c_1,\cdots, c_s$ be distinct real numbers ($0 \leq c_i \leq 1$). For given $(\phi^n, q^n)$, the collocation polynomials $u(t)$ and $v(t)$ is two polynomials of degree $s$ satisfying
\ben
    u(t_n) = \phi^n, \quad     v(t_n) = q^n,  \\
    \partial_t u(t_n^i) = \cG\left( \cL u(t_n^i) + \frac{v(t_n^i) g'\big(u(t_n^i)\big)}{\sqrt{\Big(g\big(u(t_n^i)\big), 1\Big)+C_0}} \right), \\
    \partial_t v(t_n^i) = \left( \frac{g'\big(u(t_n^i)\big)}{2\sqrt{\Big(g\big(u(t_n^i)\big), 1\Big)+C_0}}, \partial_t u(t_n^i) \right),
\een
where $t_n^i = t_n + c_i \Delta t$ and $i=1,\cdots,s.$ And then the numerical solution is defined by $\phi^{n+1} = u(t_n+\Delta t)$ and $q^{n+1} =v(t_n+\Delta t)$.
\end{scheme}
Theorem 1.4 on page 31 of \cite{HairerBook} indicates that the collocation method yields a special RK method. If the collocation points $c_1,\cdots,c_s$ are chosen as the Gaussian quadrature nodes, i.e., the zeros of the $s$-th shifted Legendre polynomial $\frac{d^s}{dx^s} \Big(  x^s (x-1)^s \Big),$  Scheme \ref{eq:Collocation-Scheme} is called the Gaussian collocation method. Based on the Gaussian quadrature nodes, the interpolating quadrature formula has order $2s$, and the Gaussian collocation method shares the same order $2s$. For instance, the RK coefficients of fourth order and sixth order HSAV schemes are given explicitly below (see \cite{HairerBook} for coefficients of higher orders).
\begin{table}[H]
\centering
\begin{tabular}{c|cc}
$\frac{1}{2}- \frac{\sqrt{3}}{6}$  $\frac{1}{4}$  $\frac{1}{4}-\frac{\sqrt{3}}{6}$  \\
$\frac{1}{2}+ \frac{\sqrt{3}}{6}$  $\frac{1}{4}+\frac{\sqrt{3}}{6}$   $\frac{1}{4}$ \\
\hline
 $\frac{1}{2}$  $\frac{1}{2}$ \\
\end{tabular}
\hspace{0.5in}
\begin{tabular}{c|ccc}
$\frac{1}{2}- \frac{ \sqrt{15}}{10}$  $\frac{5}{36} $  $\frac{2}{9} - \frac{ \sqrt{15}}{15}$  $\frac{5}{36}-\frac{\sqrt{15}}{30}$  \\
$\frac{1}{2}$  $\frac{5}{36} +\frac{ \sqrt{15}}{24}$  $\frac{2}{9} $  $\frac{5}{36}-\frac{\sqrt{15}}{24}$ \\
$\frac{1}{2}+ \frac{ \sqrt{15}}{10}$  $\frac{5}{36} +\frac{\sqrt{15}}{30} $  $\frac{2}{9} + \frac{ \sqrt{15}}{15}$  $\frac{5}{36}$  \\
\hline
 $\frac{5}{18}$  $\frac{4}{9}$  $\frac{5}{18}$ \\
\end{tabular}
\caption{RK coefficients of Gaussian collocation methods of order 4 and 6.}\label{tab-4-6-GM}
\end{table}

For conservative systems with quadratic invariants, the Gaussian collocation methods have been proven to conserve the corresponding discrete quadratic invariants \cite{HairerBook}. Here we show that they are also unconditionally energy stable for dissipative systems with quadratic free energy.

\begin{thm} \label{thm:High-SAV}
The $s$-stage HSAV Gaussian collocation Scheme \ref{eq:Collocation-Scheme} is unconditionally energy stable, i.e., it satisfies the following energy dissipation law
\ben\label{eq:High-EL}
 E^{n+1}-E^n \nonumber\\
= \Delta t \sum_{i=1}^s b_i \Big( \cL u(t_n^i) + \frac{v(t_n^i) g'\big(u(t_n^i)\big)}{\sqrt{\Big(g\big(u(t_n^i)\big), 1\Big)+C_0}} , \cG\Big[\cL u(t_n^i) + \frac{v(t_n^i) g'\big(u(t_n^i)\big)}{\sqrt{\Big(g\big(u(t_n^i)\big), 1\Big)+C_0}}\Big] \Big)\nonumber\\
\leq  0,
\een
where $E^{n} = \frac{1}{2}(\cL \phi^n, \phi^n) + (q^n)^2 - C_0$ and $t_n^i = t_n + c_i \Delta t$, $c_i$ ($i=1,\cdots,s$) are the Gaussian quadrature nodes, $b_i\geq 0$ ($i=1,\cdots,s$) are the Gauss-Legendre quadrature weights, $u(t),v(t)$ are the collocation polynomial of the Gaussian collocation methods.
\end{thm}

\begin{proof}
Noticing $\phi^n = u(t_n), q^n=v(t_n)$ and $\phi^{n+1} = u(t_{n+1}), q^{n+1}=v(t_{n+1})$, we have
\benl
E^{n+1}-E^n =   \frac{1}{2}(\cL \phi^{n+1}, \phi^{n+1}) - \frac{1}{2}(\cL \phi^n, \phi^n) + (q^{n+1})^2-(q^n)^2 \nonumber\\
= \frac{1}{2}\big( u(t_{n+1}), \cL u(t_{n+1}) \big) - \frac{1}{2}\big( u(t_{n}), \cL u(t_{n}) \big) + |v(t_{n+1})|^2 - |v(t_n)|^2 \nonumber\\
= \int_{t_n}^{t_{n+1}} \left[ \frac{1}{2}\frac{d}{d t} \big(u(t), \cL u(t)\big) + \frac{d}{dt}|v(t)|^2 \right] dt  \nonumber\\
=  \int_{t_n}^{t_{n+1}} \Big[ \big(\dot{u}(t),\cL u(t)\big) + 2\dot{v}(t) v(t) \Big] d t.
\eenl
The integrand $\Big(\dot{u}(t),\cL u(t)\Big)$ and $\dot{v}(t) v(t)$ are polynomial of degree $2s-1$, which is integrated without error by the $s$-stage Gaussian quadrature formula. It therefore follows from the collocation condition that
\benl
\int_{t_n}^{t_{n+1}} \Big[ \big(\dot{u}(t),\cL u(t)\big) + 2\dot{v}(t) v(t) \Big]d t  \\
=   \Delta t \sum_{i=1}^s b_i \Big[ \big(\dot{u}(t_n^i),\cL u(t_n^i)\big) + 2\dot{v}(t_n^i) v(t_n^i) \Big]\\
=   \Delta t \sum_{i=1}^s b_i \Big (\dot{u}(t_n^i), \cL u(t_n^i) + \frac{v(t_n^i) g'\big(u(t_n^i)\big)}{\sqrt{\Big(g\big(u(t_n^i)\big), 1\Big)+C_0}} \Big)\\
    = \Delta t \sum_{i=1}^s b_i \Big( \cL u(t_n^i) + \frac{v(t_n^i) g'\big(u(t_n^i)\big)}{\sqrt{\Big(g\big(u(t_n^i)\big), 1\Big)+C_0}} , \cG\Big[\cL u(t_n^i) + \frac{v(t_n^i) g'\big(u(t_n^i)\big)}{\sqrt{\Big(g\big(u(t_n^i)\big), 1\Big)+C_0}}\Big] \Big)\leq  0,
\eenl
which leads to \eqref{eq:High-EL}. This completes the proof.
\end{proof}

\begin{rem}
The proposed high-order energy stable schemes don't depend on the specific form of the mobility $\cG$ and the effective free energy $F$, i.e., they work for all the gradient flow models \eqref{eq:gradeint_flow}.
\end{rem}

\begin{rem}
At each time step, even though solving an HSAV scheme takes longer than solving the SAV scheme, much larger time step size could be used for the HSAV scheme than the SAV scheme to reach the same accuracy (due to the high-order accuracy of the HSAV scheme). Overall, for simulations reaching similar accuracy, the HSAV scheme will take less CPU time than the SAV scheme, making the HSAV scheme superior for long time dynamic simulations.
\end{rem}

\section{Spatial discretization}\label{sec:spaceDis}
To make the order of accuracy in space compatible with the arbitrarily high-order in time, we employ the Fourier pseudospectral method in space for Scheme \ref{scheme:RK-method} and Scheme \ref{eq:Collocation-Scheme} to arrive at fully discrete HSAV-RK methods and fully discrete HSAV collocation methods. Then the fully discrete HSAV-RK methods with \eqref{eq:RK-stable-condition} and the fully discrete HSAV Gaussian collocation methods can be proved similarly to preserve the corresponding energy dissipation law in the fully discrete level.

Firstly, we recall the two-dimensional Fourier pseudospectral method \cite{Chen2001Multi,Gong2018Linear}. To make the paper self-explanatory, we briefly reintroduce the following notations (see \cite{Gong2018Linear} for more details). Let $N_x,N_y$ be two positive even integers. The spatial domain $\Omega = [0,L_x]\times[0,L_y]$ is uniformly partitioned with mesh size $h_x =
L_{x}/N_{x},h_y = L_{y}/N_{y}$ and $$\Omega_{h} =
\left\{(x_{j},y_{k})|x_{j} = j h_x,y_{k} = kh_y,~0\leq j\leq
N_{x}-1,0\leq k\leq N_{y}-1\right\}.$$ Let $V_{h} = \big\{u|u=\{u_{j,k}|(x_{j},y_{k})\in \Omega_{h}\} \big\}$
be the space of grid functions on $\Omega_{h}$. For any two vector grid functions $\bu = (u_m),\bv = (v_m)$ $(u_m,v_m\in V_{h})$, define the discrete inner product and norm as follows
$$(\bu,\bv)_{h} = h_x h_y\sum\limits_m\sum\limits_{j = 0}^{N_{x}-1}\sum\limits_{k = 0}^{N_{y}-1}
(u_m)_{j,k}(v_m)_{j,k},~~\|\bu\|_{h} = \sqrt{(\bu,\bu)_{h}}.$$
We define
\begin{equation*}
S_{N} = \textrm{span}\{ X_{j}(x) Y_{k}(y), j =
0,1,\ldots,N_{x}-1; k = 0,1,\ldots,N_{y}-1\}
\end{equation*}
as the interpolation space, where $X_{j}(x)$ and $Y_{k}(y)$ are
trigonometric polynomials of degree $N_{x}/2$ and $N_{y}/2$, given
respectively by
\ben
X_{j}(x) = \frac{1}{N_{x}}\sum\limits_{m =
-N_{x}/2}^{N_{x}/2}{\frac{1}{a_{m}}e^{im\mu_{x}(x-x_{j})}},\\
Y_{k}(y) = \frac{1}{N_{y}}\sum\limits_{m =
-N_{y}/2}^{N_{y}/2}{\frac{1}{b_{m}}e^{im\mu_{y}(y-y_{k})}},
\een
where $$a_{m} =
\begin{cases}
1, |m|<N_{x}/2,\\
2, |m|=N_{x}/2,\\
\end{cases}
b_{m} =
\begin{cases}
1, |m|<N_{y}/2,\\
2, |m|=N_{y}/2,\\
\end{cases}
$$
and $\mu_{x} = 2\pi/L_{x}, \mu_{y} = 2\pi/L_{y}$. We define the interpolation operator $I_{N}:C(\Omega)\rightarrow
S_{N}$ as follows:
\begin{equation}\label{interpolation}
I_{N}u(x,y) = \sum\limits_{j = 0}^{N_{x}-1}\sum\limits_{k
= 0}^{N_{y}-1} u_{j,k} X_{j}(x) Y_{k}(y),
\end{equation}
where $u_{j,k} = u(x_{j},y_{k})$. The key of spatial Fourier pseudospectral discretization is to obtain derivative $\partial_{x}^{s_{1}}\partial_{y}^{s_{2}}I_{N}u(x,y)$ at
collocation points. Then, we differentiate \eqref{interpolation} and evaluate the resulting expressions at point $(x_{j},y_{k})$ as follows
\begin{equation*}
\partial_{x}^{s_{1}}\partial_{y}^{s_{2}}I_{N}u(x_{j},y_{k}) = \sum\limits_{m_1 =
0}^{N_{x}-1}\sum\limits_{m_2
= 0}^{N_{y}-1} u_{m_1,m_2} (\bD_{s_{1}}^{x})_{j,m_1}
(\bD_{s_{2}}^{y})_{k,m_2},
\end{equation*}
where $\bD_{s_{1}}^{x}$ and $\bD_{s_{2}}^{y}$ are $N_{x} \times N_{x}$ and $N_{y} \times N_{y}$ matrices, respectively, with elements given by
\begin{equation*}
(\bD_{s_{1}}^{x})_{j,m} = \frac{d^{s_{1}}X_{m}(x_{j})}{dx^{s_{1}}},~(\bD_{s_{2}}^{y})_{k,m}
=\frac{d^{s_{2}}Y_{m}(y_{k})}{dy^{s_{2}}}.
\end{equation*}
Define three operators
$\odot$, $\textcircled{x}$ and $\textcircled{y}$ as follows:
\begin{equation*}
(u\odot v)_{j,k} = u_{j,k} v_{j,k}, ~ (\bA \textcircled{x} u)_{j,k} =
\sum\limits_{m=0}^{N_{x}-1}\bA_{j,m}u_{m,k}, ~ (\bB \textcircled{y} u)_{j,k} =
\sum\limits_{m=0}^{N_{y}-1}\bB_{k,m}u_{j,m},
\end{equation*}
where $u,v\in V_h$. It is easy to show that these three operators possess the following properties:
\begin{equation*}
u\odot v = v\odot u, ~ \bA \textcircled{x} \bB \textcircled{y} u = \bB \textcircled{y} \bA \textcircled{x} u, ~ \bA \textcircled{a} \bB \textcircled{a} u = (\bA\bB) \textcircled{a} u, ~\textcircled{a} = \textcircled{x} ~\textrm{or}~ \textcircled{y}.
\end{equation*}
Then we have $$\partial_{x}^{s_{1}}\partial_{y}^{s_{2}}I_{N}u(x_{j},y_{k}) = (\bD_{s_{1}}^{x}\textcircled{x}\bD_{s_{2}}^{y}\textcircled{y}u)_{j,k}.$$

\begin{lem}[\cite{Gong2014Multi}]\label{lemFFT}
Denote
\begin{equation*}
{\bf \Lambda}_{\alpha,s} = \left\{\begin{array}{c}

\Big[i\mu_{\alpha} ph{diag}\Big(0 , 1 , \ldots , \frac{N_{\alpha}}{2}-1 , 0 , -\frac{N_{\alpha}}{2}+1 , \ldots, -1\Big)\Big]^s, ~ph{when}~s~ph{odd},\\
\Big[i\mu_{\alpha} ph{diag}\Big(0 , 1, \ldots, \frac{N_{\alpha}}{2}-1, \frac{N_{\alpha}}{2}, -\frac{N_{\alpha}}{2}+1, \ldots, -1\Big)\Big]^s, ~ph{when}~s~ph{even},
\end{array}\right.
\alpha = x~\textrm{or}~y,
\end{equation*}
we have
\begin{equation}\label{relationship-FFT}
\bD_{s}^{\alpha}=F_{N_{\alpha}}^{-1}{\bf \Lambda}_{\alpha,s}F_{N_{\alpha}},
\end{equation}
where $F_{N_{\alpha}}$ is the discrete Fourier transform, and $F_{N_{\alpha}}^{-1}$ is the discrete
inverse Fourier transform.
\end{lem}

\begin{lem}
For real matrix $\bA\in \mathbb{R}_{N_{ph{a}}\times N_{ph{a}}},ph{a} = x~\textrm{or}~y,$ and $u,v\in V_{h}$,
\begin{equation}\label{inner property}
(\bA\textcircled{a} u,v)_{h} = (u,\bA^{T}\textcircled{a}v)_{h}.
\end{equation}
\end{lem}
Using identity \eqref{inner property}, anti-symmetry of $\bD_{2s-1}^{\textrm{a}}$ and symmetry of $\bD_{2s}^{\textrm{a}},$ $\forall\textrm{a}\in\{x,y\}, s\in\mathbb{Z}^+$, we obtain
\begin{equation*}
\Big(\bD_{2s-1}^{\textrm{a}}\textcircled{a}u,v\Big)_h = - \Big(u,\bD_{2s-1}^{\textrm{a}}\textcircled{a}v\Big)_h,\quad \Big(\bD_{2s}^{\textrm{a}}\textcircled{a}u,v\Big)_h = \Big(u,\bD_{2s}^{\textrm{a}}\textcircled{a}v\Big)_h,
\end{equation*}
which implies that the Fourier pseudospectral method preserves discrete integration-by-parts formulae. Here we note that the retention of discrete integration-by-parts formulae is the key to constructing the spatial structure-preserving algorithm because the properties of the operators $\cL$ and $\cG$ are defined by the integration-by-parts formulae. Therefore, we can apply the Fourier pseudospectral method to obtain the corresponding discrete self-adjoint operator $\cL_h$ and the negative semi-definite operator $\cG_h$, i.e., they satisfy
\beq
(\cL_h\phi,\psi)_h = (\phi,\cL_h\psi)_h, \quad (\psi,\cG_h\psi)_h\leq 0,\quad \forall \phi,\psi\in V_{h}.
\eeq

Applying the Fourier pseudospectral method for Scheme \ref{scheme:RK-method}, we obtain the following fully discrete scheme.

\begin{scheme}[Fully Discrete HSAV-RK Method] \label{scheme:FP-RK-method}
Let $b_i$, $a_{ij}$ ($i,j = 1,\cdots,s$) be real numbers and let $c_i = \sum\limits_{j=1}^s a_{ij}$.
For given $\phi^n\in V_{h}$ and $q^n$, the following intermediate values are first calculated by
\beq
\bea{l}
\Phi_i =  \phi^n +  \Delta t \sum\limits_{j=1}^s a_{ij} k_j, \\
Q_i = q^n +  \Delta t \sum\limits_{j=1}^s a_{ij} l_j, \\
k_i = \cG_h \left( \cL_h \Phi_i + \frac{Q_i g'(\Phi_i) }{\sqrt{\big(g(\Phi_i),1\big)_h+C_0}} \right)  \\
l_i = \left(\frac{g'(\Phi_i)}{2 \sqrt{\big(g(\Phi_i),1\big)_h+C_0}}, k_i \right)_h,
\eea
\eeq
where $\Phi_i, k_i\in V_{h}$. Then $\phi^{n+1}\in V_{h}$, $q^{n+1}$ is updated via
\beq
\bea{l}
\phi^{n+1} = \phi^n +  \Delta t \sum\limits_{i=1}^s b_i k_i , \\
q^{n+1} = q^n +  \Delta t \sum\limits_{i=1}^s b_i l_i.
\eea
\eeq
\end{scheme}

Applying the Fourier pseudospectral method for Scheme \ref{eq:Collocation-Scheme}, we obtain the following fully discrete scheme.
\begin{scheme} [Fully Discrete HSAV Collocation Method] \label{eq:Collocation-Scheme-FP}
Let $c_1,\cdots, c_s$ be distinct real numbers ($0 \leq c_i \leq 1$). For given $\phi^n\in V_{h}$ and $q^n$, $u(t)$ is a $N_x\times N_y$ matrix polynomial of degree $s$ and $v(t)$ is a polynomial of degree $s$ satisfying
\ben
    u(t_n) = \phi^n, \quad     v(t_n) = q^n,  \\
  \dot{u}(t_n^i) = \cG_h\left( \cL_h u(t_n^i) + \frac{v(t_n^i) g'\big(u(t_n^i)\big)}{\sqrt{\Big(g\big(u(t_n^i)\big), 1\Big)_h+C_0}} \right), \\
  \dot{v}(t_n^i) = \left( \frac{g'\big(u(t_n^i)\big)}{2\sqrt{\Big(g\big(u(t_n^i)\big), 1\Big)_h+C_0}}, \dot{u}(t_n^i) \right)_h,
\een
where $t_n^i = t_n + c_i \Delta t$ and $i=1,\cdots,s.$ And then the numerical solution is defined by $\phi^{n+1} = u(t_n+\Delta t)$ and $q^{n+1} =v(t_n+\Delta t)$.
\end{scheme}

Similarly, we have the following theorems.

\begin{thm} \label{thm:High-SAV-RK-FD}
If the coefficients of a fully discrete HSAV-RK method satisfy
\beq\label{eq:RK-stable-condition-FD}
b_i a_{i j} + b_j a_{j i} = b_i b_j,\quad b_i\geq 0,\quad \forall~i,j = 1,\cdots,s,
\eeq
then it is unconditionally energy stable, i.e., it satisfies the following energy dissipation law
\beq\label{eq:High-EL-RK-FD}
E_h^{n+1}-E_h^n = \Delta t \sum_{i=1}^s b_i \Big(  \cL_h \Phi_i + \frac{ Q_i g'(\Phi_i )}{\sqrt{\big(g(\Phi_i ), 1\big)_h+C_0}} , \cG_h\Big[ \cL_h \Phi_i + \frac{ Q_i g'(\Phi_i )}{\sqrt{\big(g(\Phi_i ), 1\big)_h+C_0}} \Big] \Big)_h\leq 0,
\eeq
where $E_h^{n} = \frac{1}{2}(\cL_h \phi^n, \phi^n)_h + (q^n)^2 - C_0$.
\end{thm}

\begin{thm} \label{thm:High-SAV-FD}
The fully discrete HSAV Gaussian collocation Scheme \ref{eq:Collocation-Scheme-FP} is unconditionally energy stable, i.e., it satisfies the following energy dissipation law
\ben\label{eq:High-EL-FD}
 E_h^{n+1}-E_h^n \nonumber\\
= \Delta t \sum_{i=1}^s b_i \Big( \cL_h u(t_n^i) + \frac{v(t_n^i) g'\big(u(t_n^i)\big)}{\sqrt{\Big(g\big(u(t_n^i)\big), 1\Big)_h+C_0}} , \cG_h\Big[\cL_h u(t_n^i) + \frac{v(t_n^i) g'\big(u(t_n^i)\big)}{\sqrt{\Big(g\big(u(t_n^i)\big), 1\Big)_h+C_0}}\Big] \Big)_h\nonumber\\
\leq  0,
\een
where $E_h^{n} = \frac{1}{2}(\cL_h \phi^n, \phi^n)_h + (q^n)^2 - C_0$ and $t_n^i = t_n + c_i \Delta t$, $c_i$ ($i=1,\cdots,s$) are the Gaussian quadrature nodes, $b_i\geq 0$ ($i=1,\cdots,s$) are the Gauss-Legendre quadrature weights.
\end{thm}

As the proofs of Theorem \ref{thm:High-SAV-RK-FD} and \ref{thm:High-SAV-FD} are similar with their semi-discrete version as shown in Theorem \ref{thm:High-SAV-RK} and \ref{thm:High-SAV}, we thus omit the details for simplicity.

\section{Numerical examples}\label{sec:Numer}

In this section, we conduct several numerical tests to verify the theoretical results in the previous section. We emphasize that the newly proposed HSAV schemes could reach arbitrarily high order accuracy in time (with proper choice of the Gaussian collocation points), and they are all unconditionally energy stable. For simplicity, in the rest of this paper, we only use 4th and 6th order for demonstration purpose. Moreover, the CPU time is calculated with a 3.2 GHz Intel Core i7 using Matlab R2018b on MacOS Mojave version 10.14.2.

\textbf{Example 1: the Allen-Cahn equation.} First of all, we test the proposed numerical schemes for solving the widely-used Allen-Cahn (AC) equation \cite{AC}. Mainly, the AC equation is proposed as
\beq
\partial_t \phi = - M \Big(- \varepsilon^2 \Delta \phi + (\phi^3-\phi)  \Big),
\eeq 
where $M$ is the mobility parameter and $\varepsilon$ controls the interfacial thickness. We choose the broadly embraced benchmark problem \cite{ChenL2002}, i.e. set the initial profile for $\phi$ as:
\beq
\phi(x,y,t=0) = \left\{
\bea{l}
1, \quad x^2 + y^2 < 100^2, \\
-1, \quad x^2 + y^2 \geq 100^2,
\eea 
\right.
\eeq 
which is a disk centered at the origin, and use the domain $[-128 \,\,\,\ 128]^2$. The parameters are chosen as $M=\varepsilon=1$. It is known that the area of the disk will shrink, following the linear dynamics $V=\pi R_0^2 - 2\pi t$ asymptotically, with $R_0$ the initial radius. Here we test the dynamics using the proposed HSAV schemes. The numerical results are summarized in Figure \ref{fig:AC-Circle}. We observe that the HSAV schemes can use much larger time step to capture the correct volume shrinking dynamics than the classical SAV schemes.
\begin{figure}
\center
\subfigure[SAV-CN Scheme]{\includegraphics[width=0.3\textwidth]{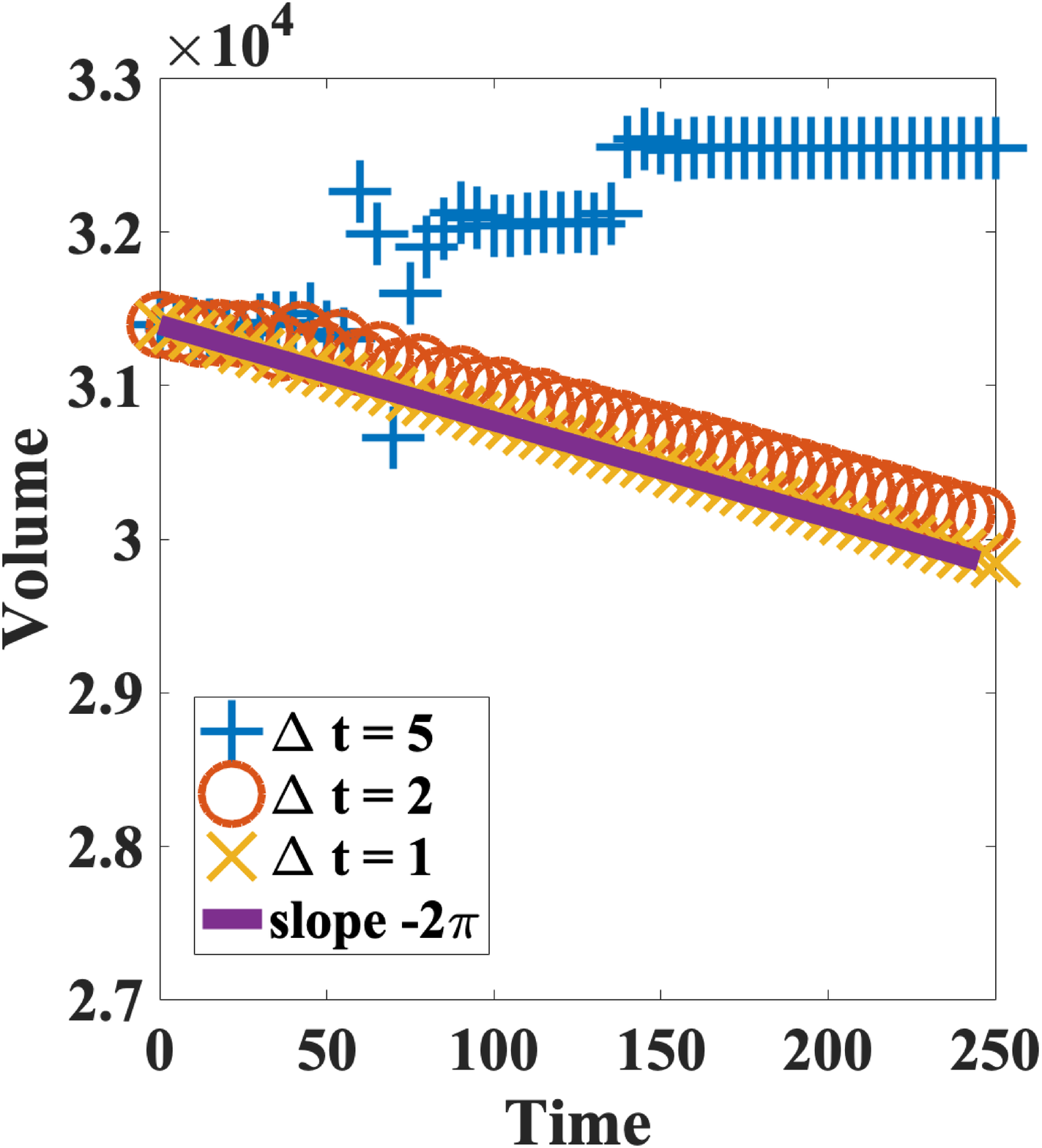}}
\subfigure[HSAV 4th-order Scheme]{\includegraphics[width=0.3\textwidth]{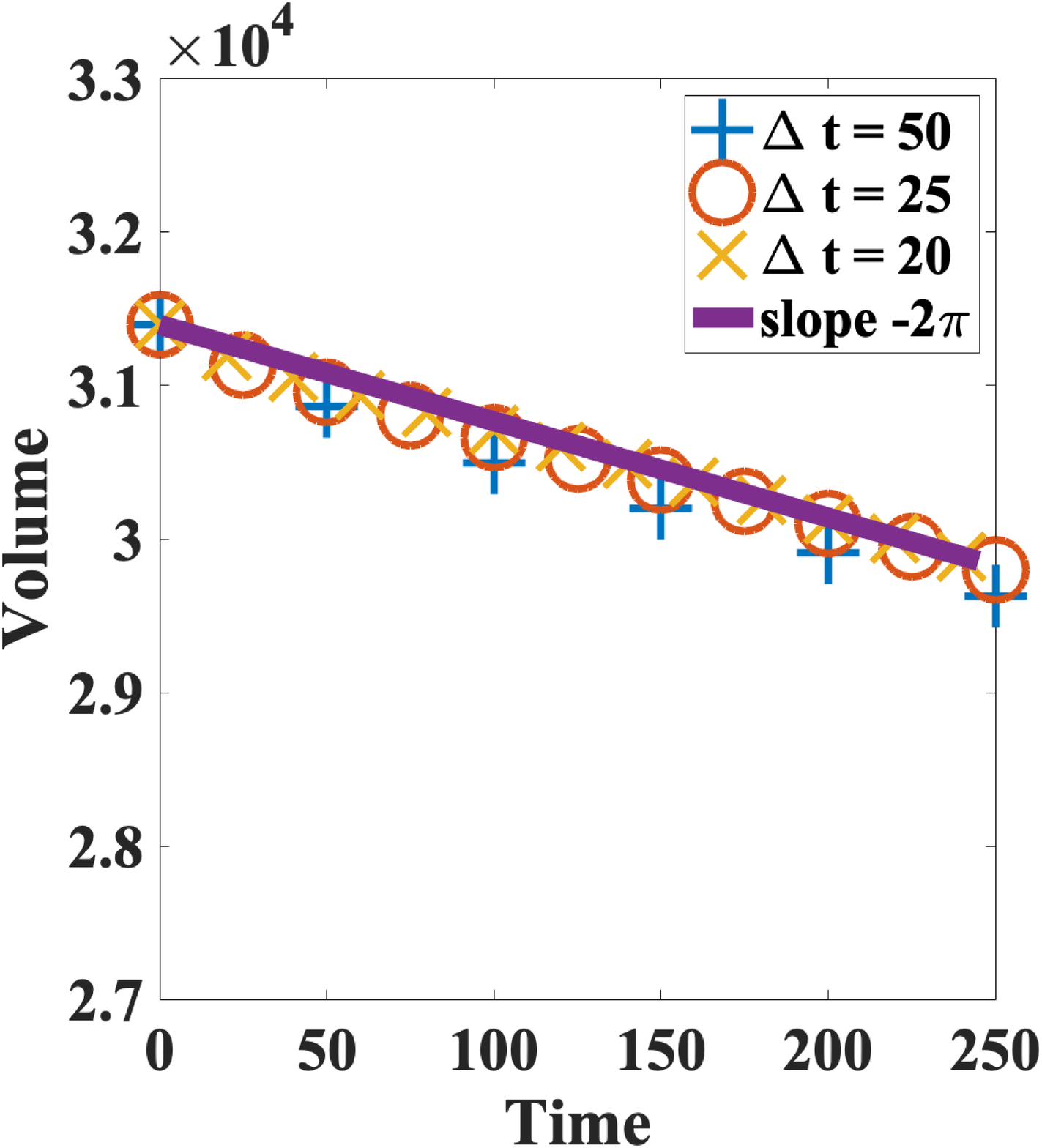}}
\subfigure[HSAV 6th-order Scheme]{\includegraphics[width=0.3\textwidth]{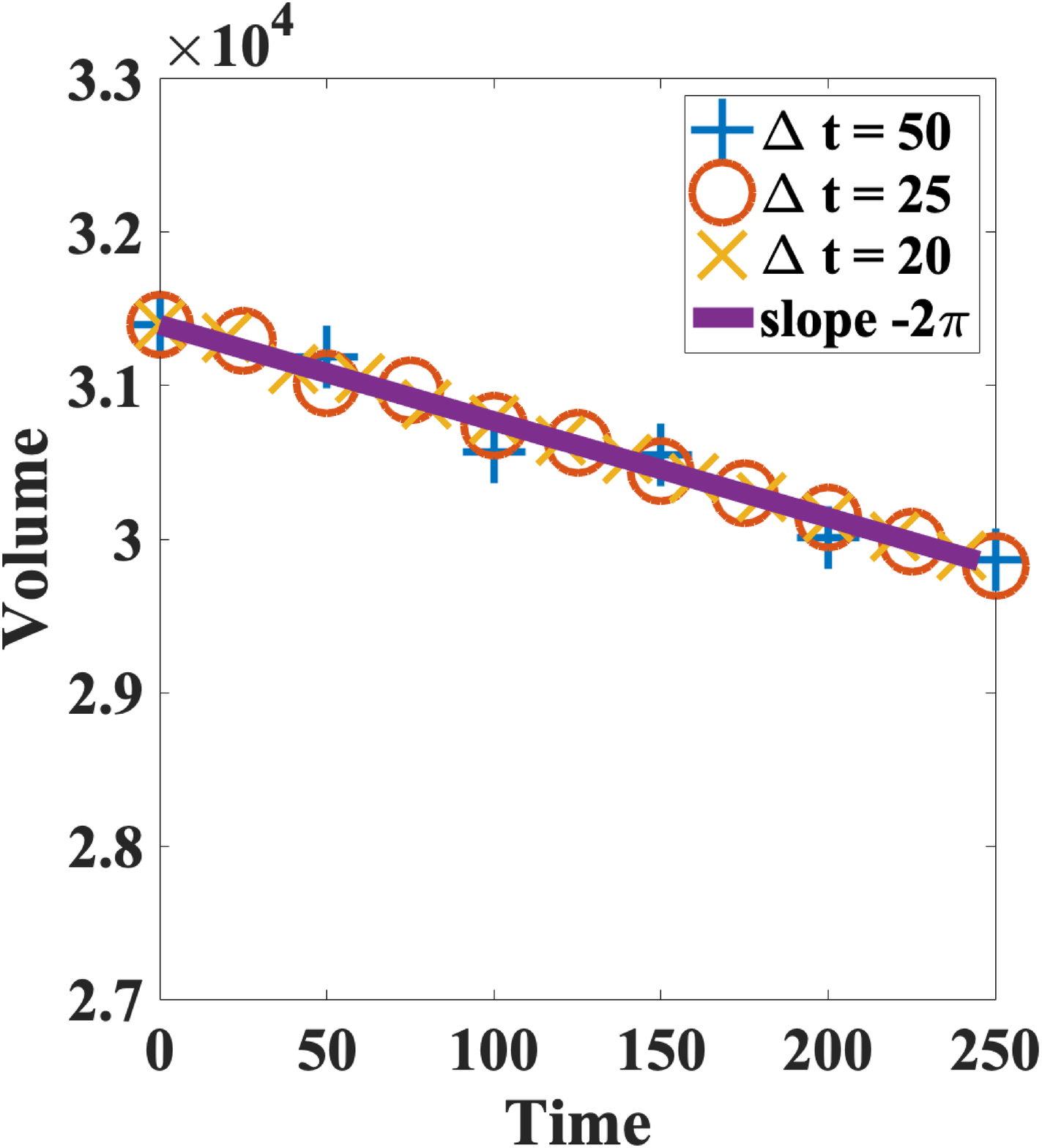}}
\caption{Benchmark problem for the Allen-Cahn equation. This figure shows the volume of the disk decreasing with time with different schemes and various time steps.}
\label{fig:AC-Circle}
\end{figure}

\textbf{Example 2: the Cahn-Hilliard equation.} Next, we study the widely-used Cahn-Hilliard equation with the Ginzburg-Landau free energy.  Specifically, given the Ginzburg-Landau free energy $F=(-\frac{\varepsilon^2}{2}\Delta \phi, \phi) + (\frac{1}{4}(1-\phi^2)^2,1)$ and constant mobility $\lambda$, the model is proposed as
\beq \label{eq:CH}
\partial_t \phi = \lambda \Delta \Big[  -\varepsilon^2 \Delta \phi + (\phi^3-\phi) \Big] .
\eeq
If we set $\cG= \lambda \Delta$, $\cL= -\varepsilon^2 \Delta + \gamma_0$ and $g(\phi) = \frac{1}{4}(1-\phi^2)^2 - \gamma_0 \phi^2 +\frac{C_0}{|\Omega|}$, where $C_0$ is a constant such that $(g,1)>0$. By introducing the scalar auxiliary variable $q=\sqrt{(g,1)}$, the Cahn-Hilliard equation \eqref{eq:CH} could be rewritten as the reformulated gradient flow form of \eqref{eq:gradient_flow_SAV}
\beq
\bea{l}
\partial_t \phi = \lambda \Delta \Big[ -\varepsilon^2 \Delta \phi + \gamma_0 \phi  + \frac{q}{\sqrt{(g,1)}}g' \Big], \\
\partial_t q = ( \frac{g'}{2\sqrt{(g,1)}}, \partial_t \phi),
\eea
\eeq
with the consistent initial condition for $q$, i.e. $q(t=0)= \sqrt{ (g, 1)}|_{t=0}$.

First of all, we conduct a time-step refinement test to verify the accuracy of our proposed high order schemes.  We choose the domain as $[0,1] \times [0, 1]$ and spatial meshes $N_x=N_y=256$. The parameters are chosen as $\lambda=10^{-3}$, $\epsilon=0.01$, $\gamma_0=1$, $C_0=1$. The initial profile for $\phi$ is given as $\phi(x,y,t=0)=\sin(2\pi x)\sin(2\pi y)$. Both the SAV Crank-Nicolson (SAV-CN) scheme (see \cite{SAV-1,SAV-2}) and the newly proposed HSAV Scheme \ref{eq:Collocation-Scheme} with fourth order and sixth order collocation points are tested. The numerical errors in $L^2$ norm at $t=1$ are summarized in Figure \ref{fig:CH-order}. We observe that the two HSAV schemes reach the fourth and sixth order accuracy respectively. In particular, the $L^2$ errors of HSAV schemes are significantly (in several orders of magnitudes) smaller than the SAV-CN scheme.

\begin{figure}[H]
\center
\includegraphics[width=0.7\textwidth]{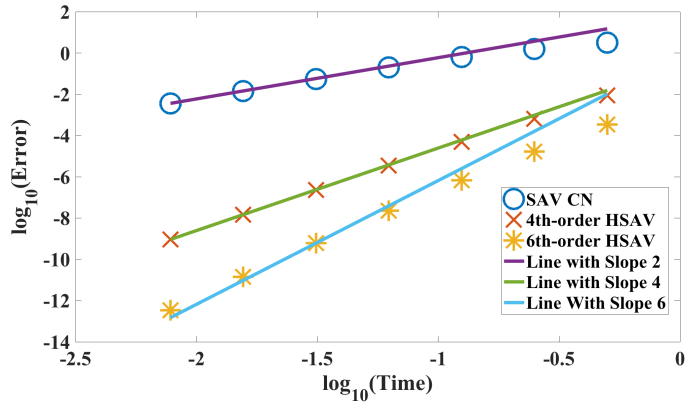}
\caption{Time refinement test for the SAV schemes solving the Cahn-Hilliard equation. This figure demonstrates the HSAV scheme can reach its high-order accuracy. And its numerical error is dramatically smaller than the SAV-CN scheme.}
\label{fig:CH-order}
\end{figure}

In addition, to assure the $L^2$ norm of the numerical error for $\phi$ at $t=1$ smaller than $10^{-10}$, the approximately minimum time steps are $\delta t=10^{-5}$ for the SAV scheme, $\delta t=0.004$ for the HSAV 4th-order scheme, $\delta t=0.02$ for the HSAV 6th-order scheme. The total CPU time is summarized in Table \ref{tab:CH-CPU}, where we observe the HSAV scheme takes much less CPU time than the SAV scheme. It indicates the HSAV schemes are superior to the SAV schemes for accurate long-time dynamic simulations.

\begin{table}[H]
\caption{ Total CPU time using various numerical schemes solving the CH model.}
\label{tab:CH-CPU}
\scriptsize
\centering
\begin{tabular}{c|c|c|c}
\hline
 SAV Scheme  HSAV 4th-order Scheme  HSAV 6th-order Scheme \\
\hline
$\delta t$   $0.00001$  $0.025$  $0.02$ \\
\hline
CPU time (seconds)  165.32   3.00   2.62 \\
\hline
\end{tabular}
\end{table}

Next we compare the different SAV schemes for  simulating the coarsening dynamics of two-phase immersible fluids using the Cahn-Hilliard equation in \eqref{eq:CH}. We choose the domain as $[0,4\pi] \times [0,4\pi]$ and use meshes $N_x=N_y=512$. The parameters are chosen as $\lambda = 0.1$, $\epsilon = 0.025$, $\gamma_0=1$, $C_0=1$. And we use an initial profile of $\phi$ as
\beq
\phi(x,y,t=0)= 0.001 rand(x,y),
\eeq
where $rand(x,y)$ generates random number between $-1$ and $1$.  The predicted energy evolution using different SAV schemes with various time steps are summarized in Figure \ref{fig:Coarsening-SAV}. We observe that for the SAV-CN scheme, it can predict the correct energy evolution with time step $\Delta t = 0.00025$ (where the predicted energy evolution with time step $\Delta t=0.0005$ is noticeably inaccurate). For the fourth order HSAV scheme, it can predict accurate energy evolution even with time step $\Delta t=0.05$; and for the sixth order HSAV scheme, it even works well with time step $\Delta t=0.1$, which is more than $10^3$ bigger than the one with the SAV-CN scheme.

\begin{figure}[H]
\center
\includegraphics[width=0.7\textwidth]{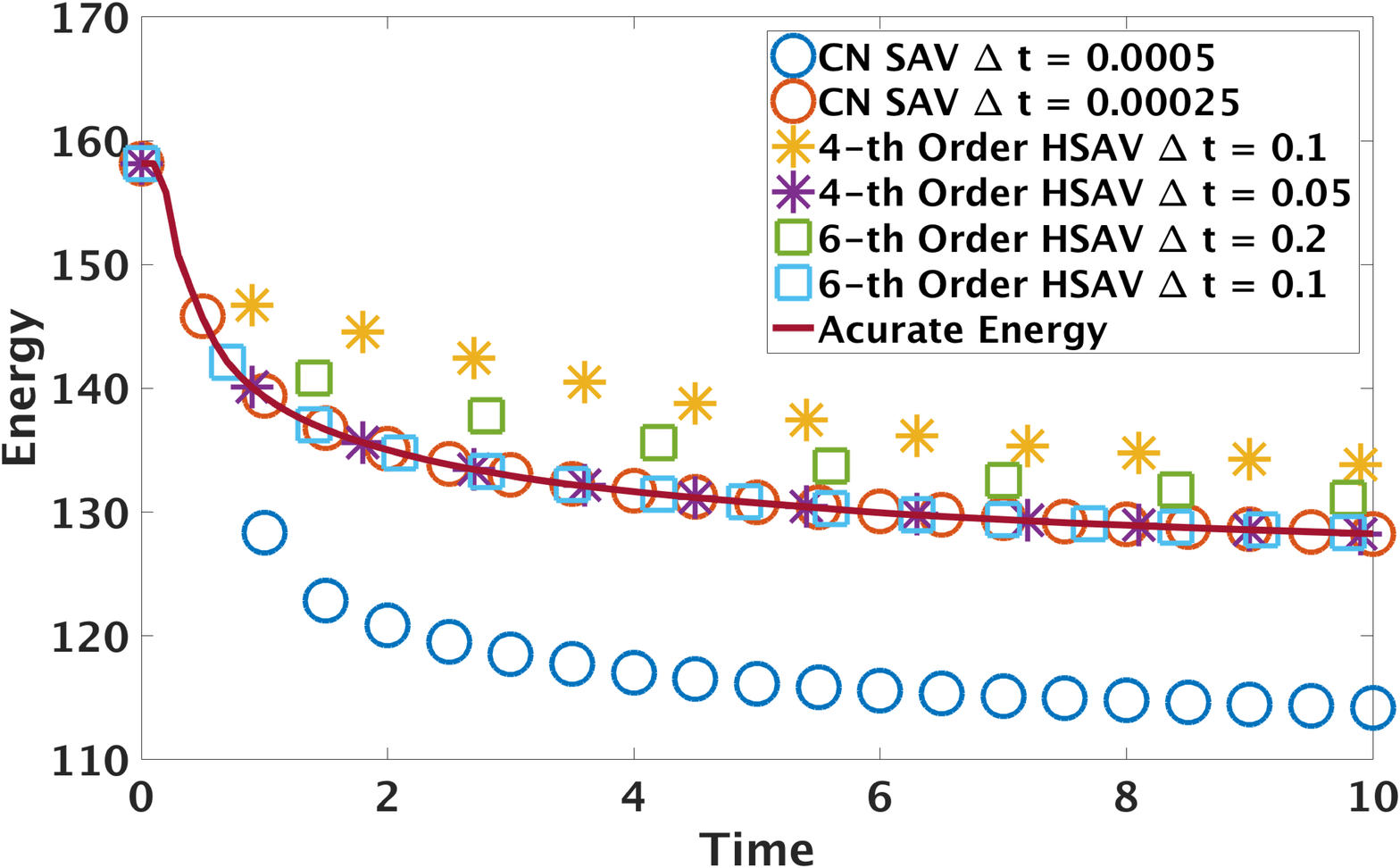}
\caption{A comparison of energy evolution using different SAV numerical schemes with various time steps. This figure illustrates the HSAV scheme could predict accurate energy evolution with much larger time steps than the  SAV-CN scheme. }
\label{fig:Coarsening-SAV}
\end{figure}

Then we use the HSAV schemes to conduct the long-time dynamic simulations of coarsening. We use the same parameters as above, and choose the initial profile
\beq
\phi(x,y,t=0) = \phi_0 + 10^{-3} rand(x,y),
\eeq
where $\phi_0$ is a constant and $rand(x,y)$ generates a random number in the range of $-1$ to $1$. Then we choose $\phi_0=0, 0.1, 0.5$. We use the sixth order HSAV scheme with the time step $\Delta t = 0.1$. The simulation results are summarized in Figure \ref{fig:Coarsening-Ex}, where we present the profile of $\phi$ at different times. We can observe the HSAV schemes can capture the phase transition dynamics accurately even with such time step time size. In particular, when $\phi_0$ is small, i.e., the two components have similar total volume,  the spinodal decomposition takes effect. When the volume of one phase is dominant, (for instance, $\phi_0=0.5$), the nucleation takes effects. These findings are a strong agreement with reports in \cite{GomezHughesJCP2011}.

\begin{figure}[H]
\center
\subfigure[]{
\includegraphics[width=0.22\textwidth]{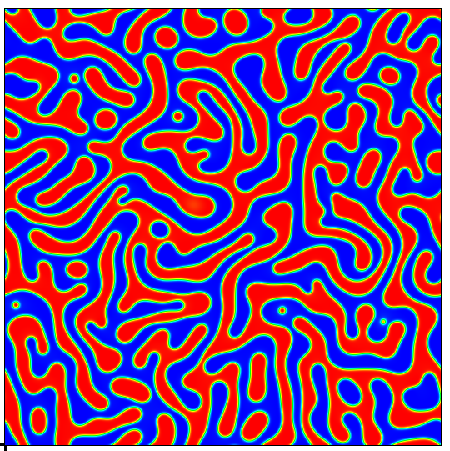}
\includegraphics[width=0.22\textwidth]{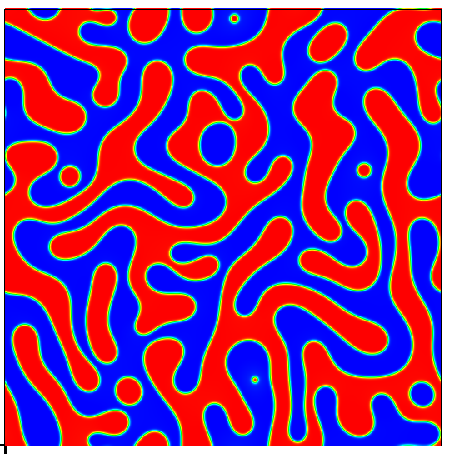}
\includegraphics[width=0.22\textwidth]{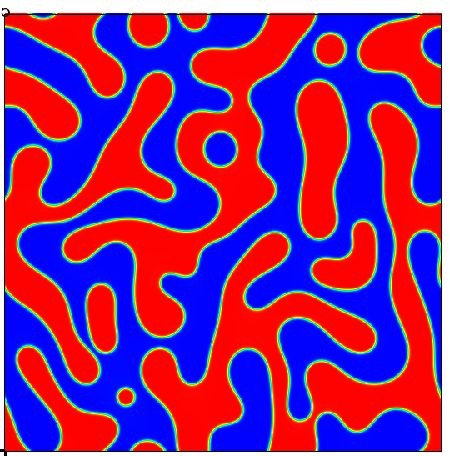}
\includegraphics[width=0.22\textwidth]{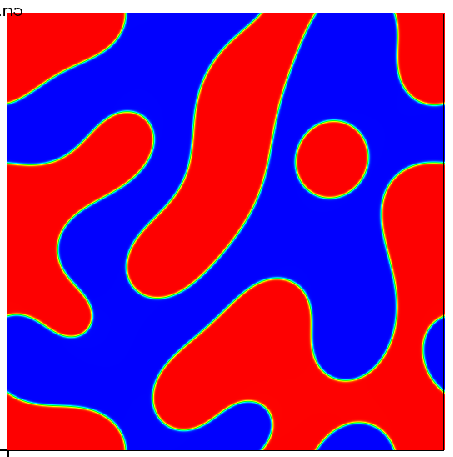}
}

\subfigure[]{
\includegraphics[width=0.22\textwidth]{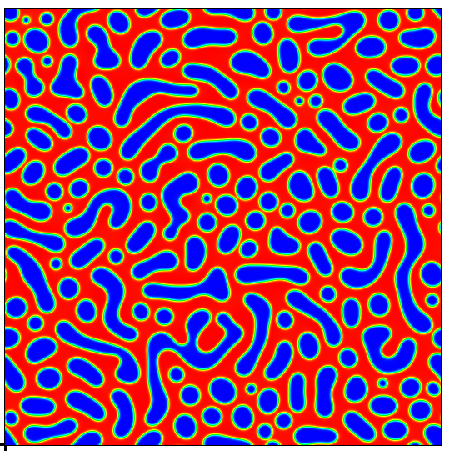}
\includegraphics[width=0.22\textwidth]{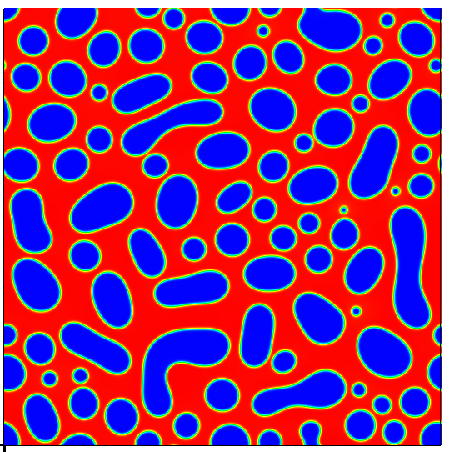}
\includegraphics[width=0.22\textwidth]{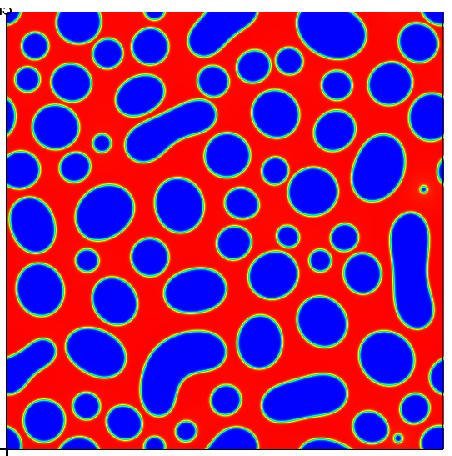}
\includegraphics[width=0.22\textwidth]{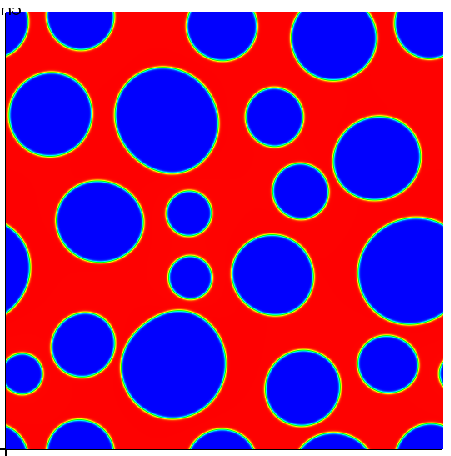}
}

\subfigure[]{
\includegraphics[width=0.22\textwidth]{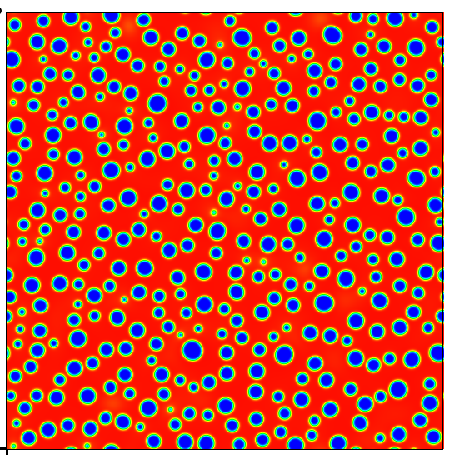}
\includegraphics[width=0.22\textwidth]{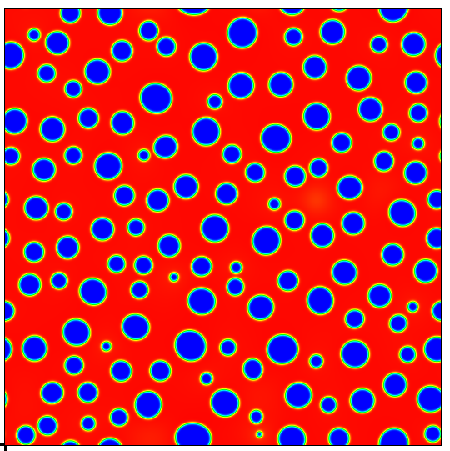}
\includegraphics[width=0.22\textwidth]{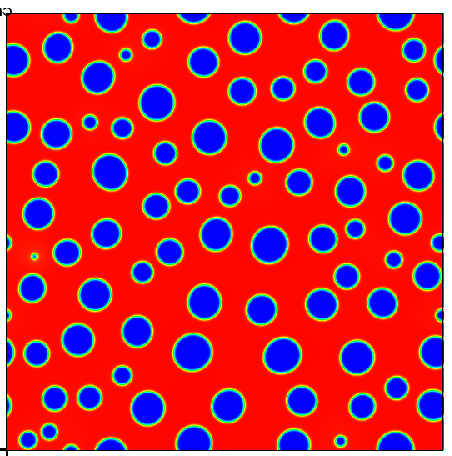}
\includegraphics[width=0.22\textwidth]{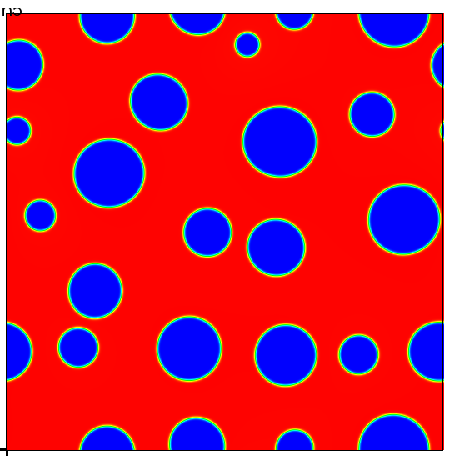}
}

\caption{Coarsening dynamics of two phase immersible fluid using the 6th order HSAV scheme with time step $\Delta t=0.1$. Here we choose $\phi_0=0,0.1,0.5$, and the results are shown in (a)-(c) respectively. This figure presents the profile of $\phi$ at time $t=10,50,100,1000$ for the simulation in Figure \ref{fig:Coarsening-SAV}. It illustrates the HSAV scheme could capture the phase separation dynamics accurately even with large time step.}
\label{fig:Coarsening-Ex}
\end{figure}

Next, we study the power-law coarsening dynamics. Here we set $\lambda=0.02$, $\varepsilon=0.05$, and domain  $[0 \,\,\, 4\pi]^2$,  and use the initial profile $\phi(x,y,t=0)=0.001 rand(-1,1)$. It is known that the effective free energy decreases asymptotically following a power law $E(t) \approx O(t^{-1/3})$. We use the 4th, 6th order HSAV  schemes and the SAV-CN scheme to calculate it, with $256^2$ meshes and $\gamma_0=1$, $C_0=1$. The obtained results are summarized in Figure \ref{fig:CN-PowerLaw}.  We observe that all the numerical schemes can capture the power law dynamics very well when the time step is small enough, saying when $\delta t=10^{-3}$. However, the maximum time step of capturing the correct dynamics using the HSAV scheme is much larger than that of the SAV-CN scheme.

\begin{figure}[H]
\center
\subfigure[SAV-CN Scheme]{\includegraphics[width=0.3\textwidth]{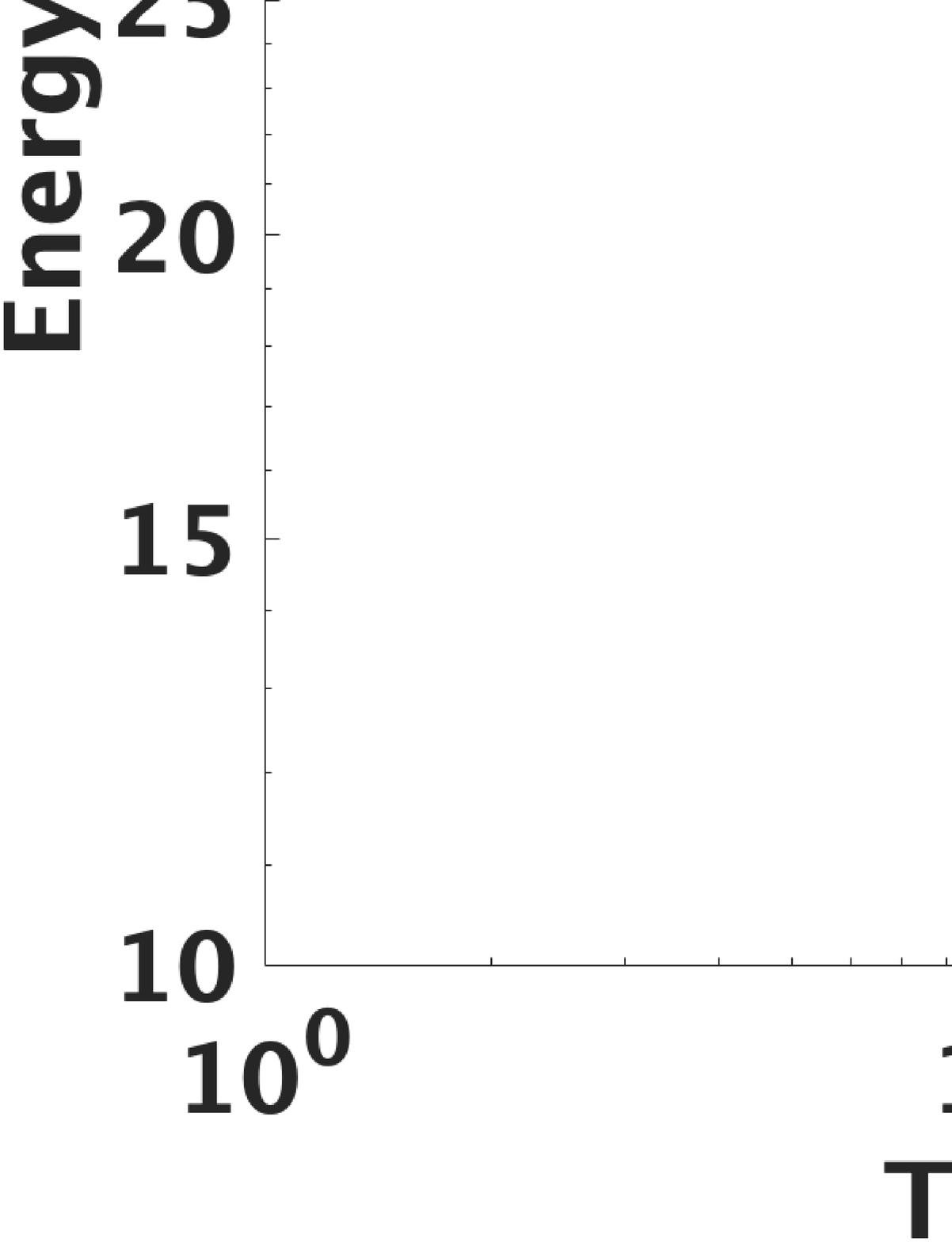}}
\subfigure[HSAV 4th order Scheme]{\includegraphics[width=0.3\textwidth]{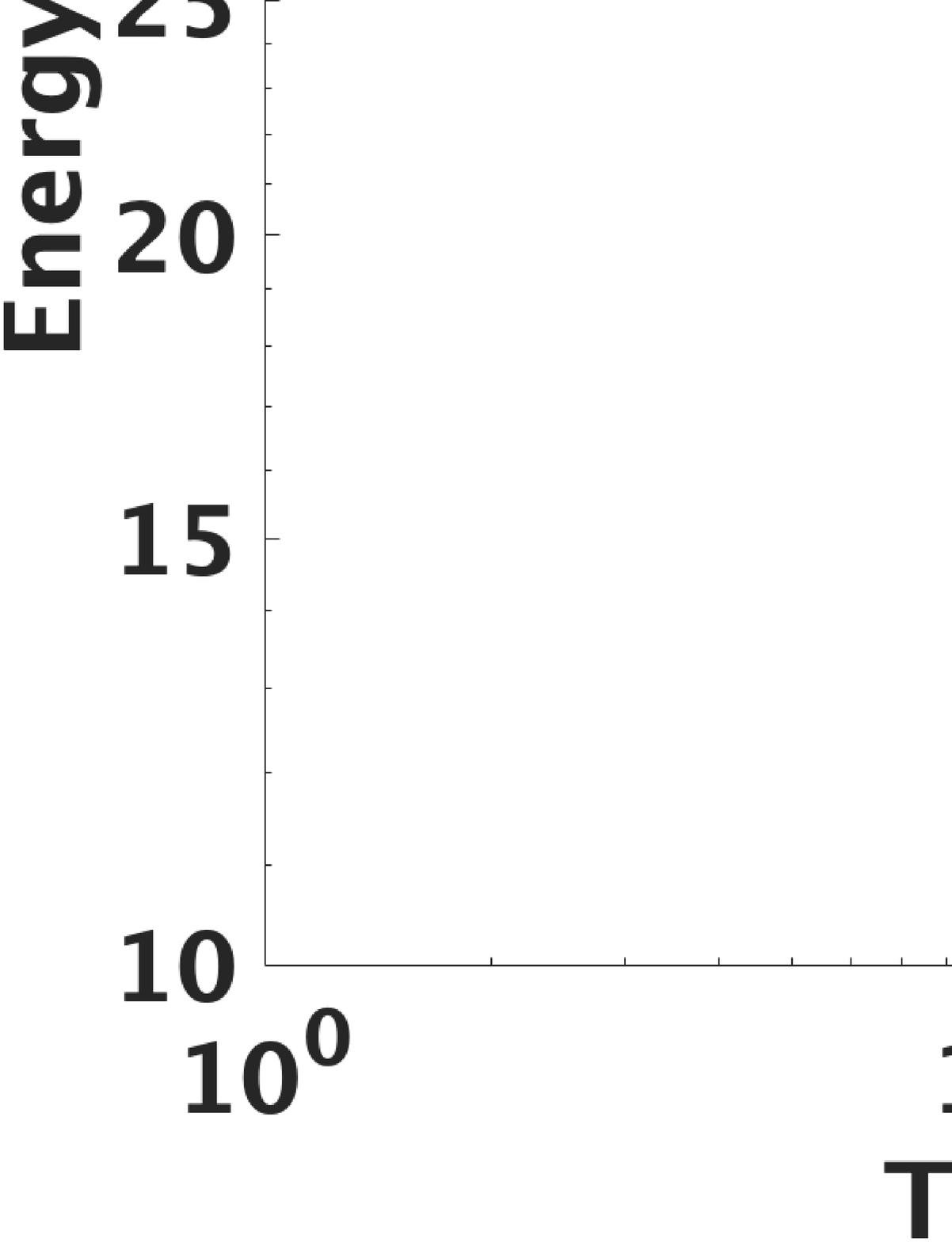}}
\subfigure[HSAV 6th order Scheme]{\includegraphics[width=0.3\textwidth]{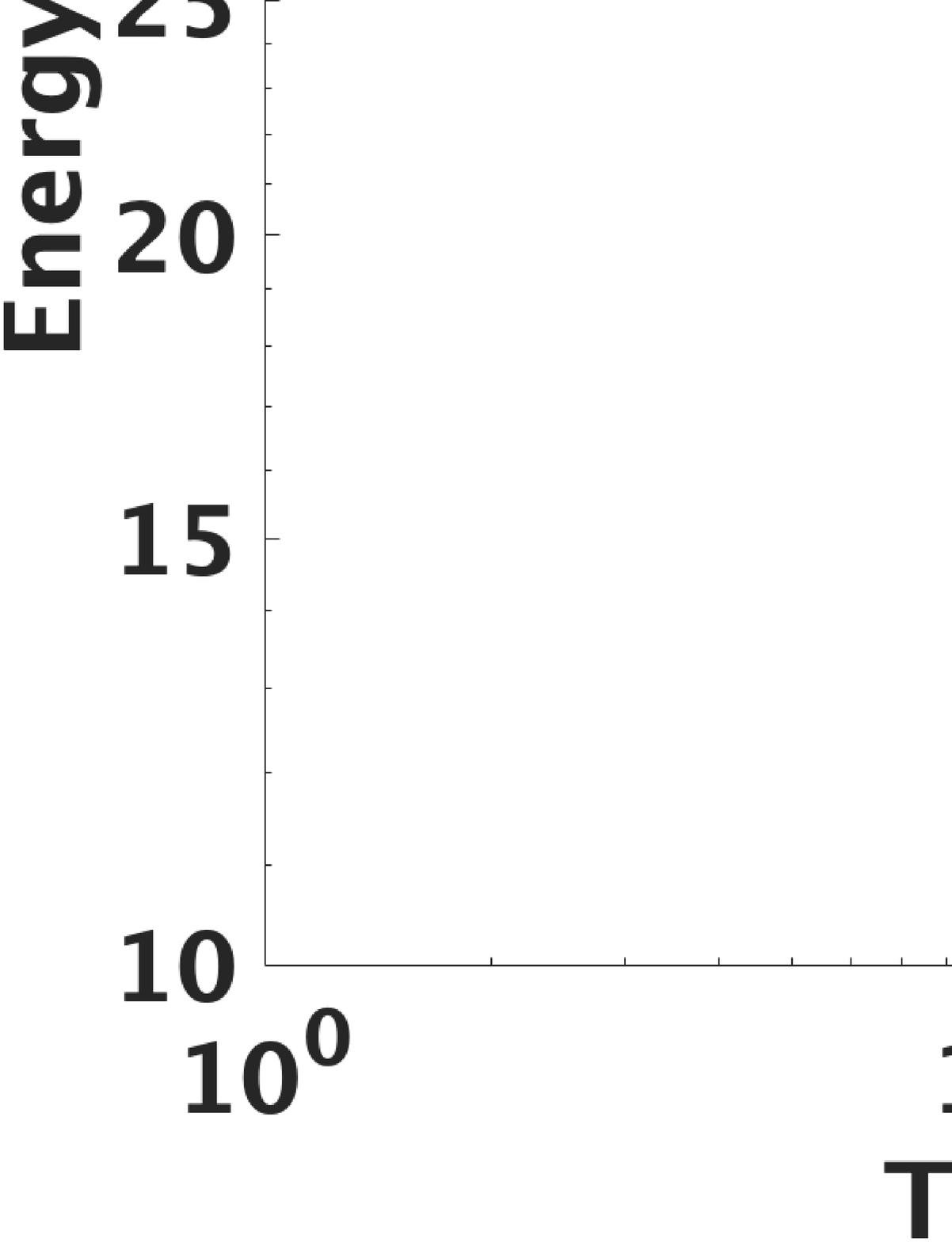}}
\caption{The energy evolution with different time steps. Here the log-log scale of the energy with respect to time is plotted.}
\label{fig:CN-PowerLaw}
\end{figure}

\textbf{Example 3: The Molecular Beam Epitaxy Model.} In the last case, we consider the molecular beam epitaxy (MBE) growth model without slop selection \cite{WangWangWiseDCDS2010}. There is a huge amount of work in literature on investigating the MBE models analytically and numerically \cite{WangWangWiseDCDS2010,ShenWangWise2012,QiaoSunZhangMC2015,QiaoWangWiseZhangIJNAM2017,ChenWangWangWiseJSC2014,FengWangWiseZhang2017,JuLiQiaoZhangMC2017,SAV-2,YangZhaoWangJCP2017,FengWangWiseMBE2017,XuSIAM2006,Kohn2003,LiLiu2003,LiLiuJNS2004}.

Given the height profile of MBE denoted as $\phi$, the evolution equation reads as
\beq
\partial_t \phi = -M \Big( \varepsilon^2 \Delta^2 \phi + \nabla \cdot ( (1-|\nabla \phi|^2)\nabla \phi) \Big),
\eeq
with periodic boundary condition. This model could be viewed as a $L^2$ gradient flow with respect to the effective free energy
\beq
F(\phi) = \int_\Omega \Big( \frac{\varepsilon^2}{2} (\Delta \phi)^2 + \frac{1}{4}( |\nabla \phi|^2 - 1)^2 \Big) d\bx,
\eeq
with a constant mobility $M$.

If we denote $\cG= -M$, $\cL= \varepsilon^2 \Delta^2  - \gamma_0 \Delta$ and $g(\nabla \phi) = \frac{1}{4} (|\nabla \phi|^2 - 1 -\gamma_0)^2 + \frac{C_0}{|\Omega|}$, and introduce the scalar auxiliary variable
\beq
q = \sqrt{(g,1)}, 
\eeq
and the intermediate function
\beq
\displaystyle
H(\nabla \phi) = \frac{  \nabla \cdot \Big(  (\gamma_0+ 1-|\nabla \phi|^2)\nabla \phi \Big)}{2\sqrt{ \int_\Omega \frac{1}{4} (|\nabla \phi|^2 - 1 -\gamma_0)^2 d\bx + C_0}},
\eeq
the equation can be reformulated as
\beq
\bea{l}
\partial_t \phi = -M \Big( \varepsilon^2 \Delta^2 \phi - \gamma_0 \Delta \phi + 2 q H \Big), \\
\partial_t q = \Big(H, \partial_t \phi \Big),
\eea
\eeq
with the consistent initial condition for $q$, i.e. $q(t=0)= \sqrt{(g,1)}|_{t=0}$.

First of all, we would like to test the convergence rate for our proposed scheme. Following the strategy in example 1, we use Cauchy sequences, where the errors are calculated as the differences between numerical solutions with adjacent time steps. Wet set $M=1$, $\varepsilon=1$ and the domain $[0 \,\,\, 2\pi]^2$. We use a smooth initial condition $\phi(x,y,0)=\sin(x)\sin (y)$, and choose $\gamma_0=1$, $C_0=1$, $256^2$ meshes.
The refinement-test results are summarized in Figure \ref{fig:MBE-order}. We observe that all the schemes reach their expected orders of convergence when the time-step is small enough. However, the HSAV schemes have dramatically smaller numerical errors (with several magnitudes smaller) than the SAV-CN scheme, which highlights the advantages of the newly proposed HSAV schemes.

\begin{figure}[H]
\center
\includegraphics[width=0.7\textwidth]{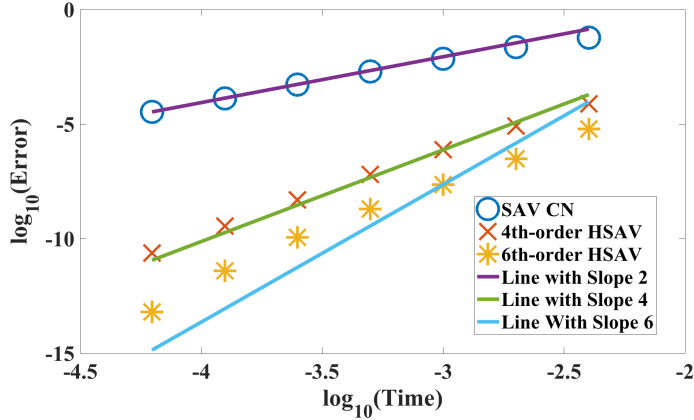}
\caption{Time refinement test for the SAV schemes. This figure demonstrates the HSAV scheme can reach its high-order accuracy. And its numerical error is dramatically smaller than the  SAV-CN schemes.}
\label{fig:MBE-order}
\end{figure}

Then, the proposed 4th-order and 6th-order  HSAV schemes are tested via a benchmark problem \cite{YangZhaoWangJCP2017}, Consider the domain $[0 , 2\pi]^2$, and  parameters $\varepsilon^2=0.1$, $M=1$. We pick the initial profile
$$
\phi(x,y,t=0)=0.1(\sin(3x)\sin(2y)+\sin(5x)\sin(5y)).
$$
This is a classic example that has been studied intensively \cite{WangWangWiseDCDS2010,ChenZhaoYang2018}.
The effective free energy dynamics using different schemes with various time steps are plotted. We notice that even though all schemes assure the energy dissipation properties, the SAV-CN scheme requires a much smaller time step size (around $\delta t=10^{-4}$) to predict accurate energy dissipation. In the meanwhile, the HSAV scheme could predict energy evolution accurately even with the time step $\delta t = 0.05$, which is $500$ larger than the SAV scheme.

\begin{figure}[H]
\center
\subfigure[SAV scheme]{\includegraphics[width=0.32\textwidth]{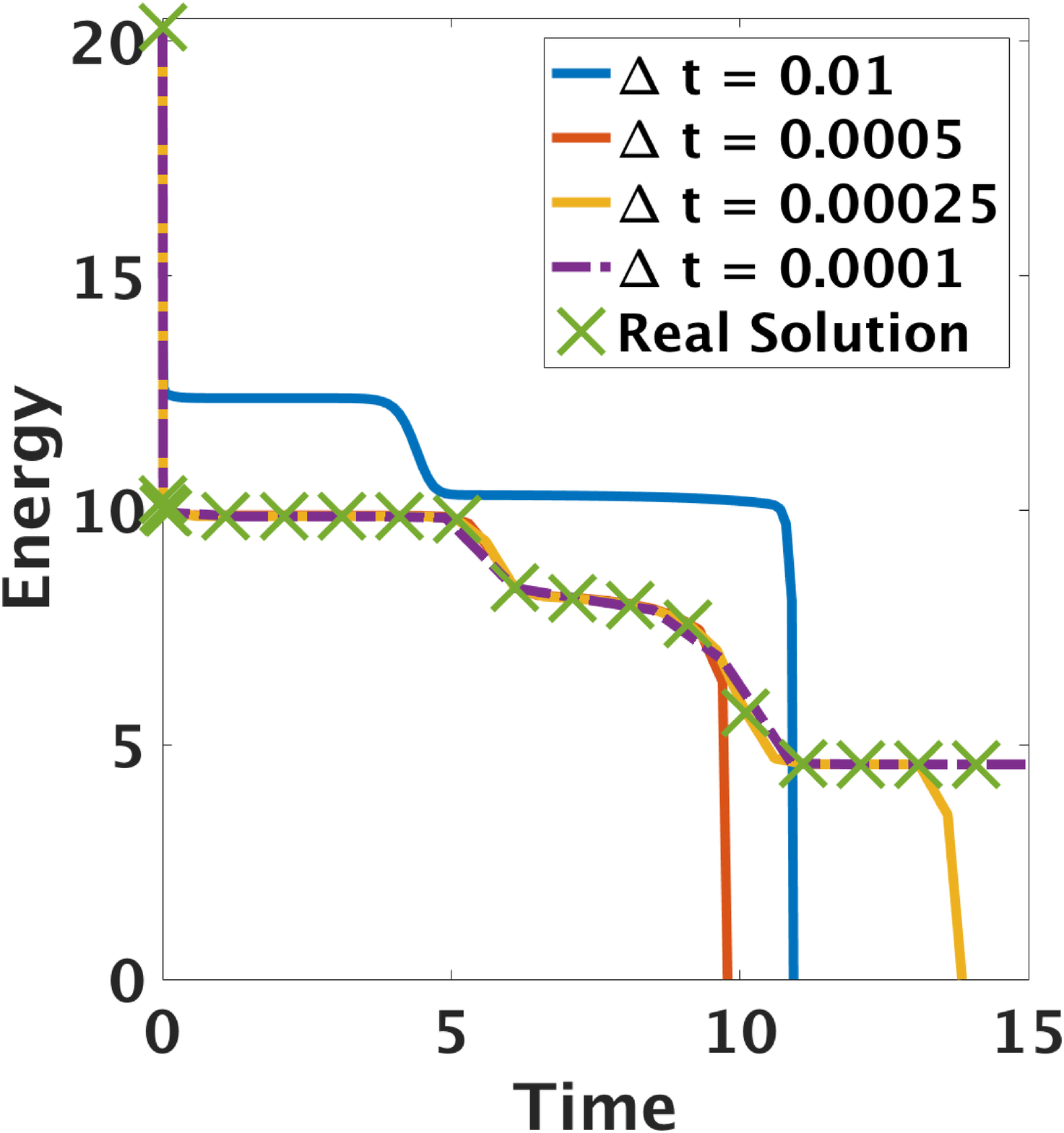}}
\subfigure[4th order HSAV scheme]{    \includegraphics[width=0.32\textwidth]{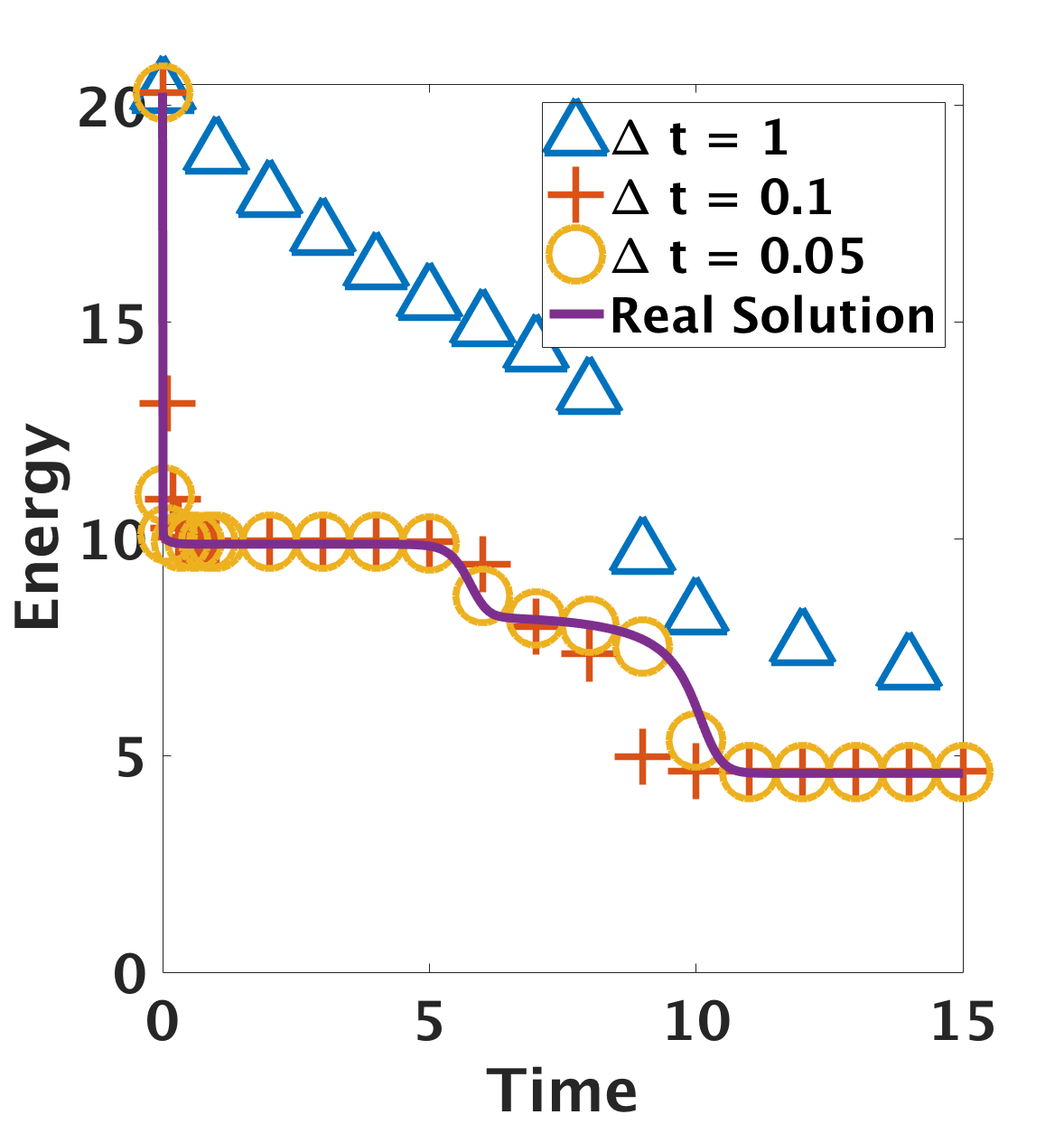}}
\subfigure[6th order HSAV scheme]{\includegraphics[width=0.32\textwidth]{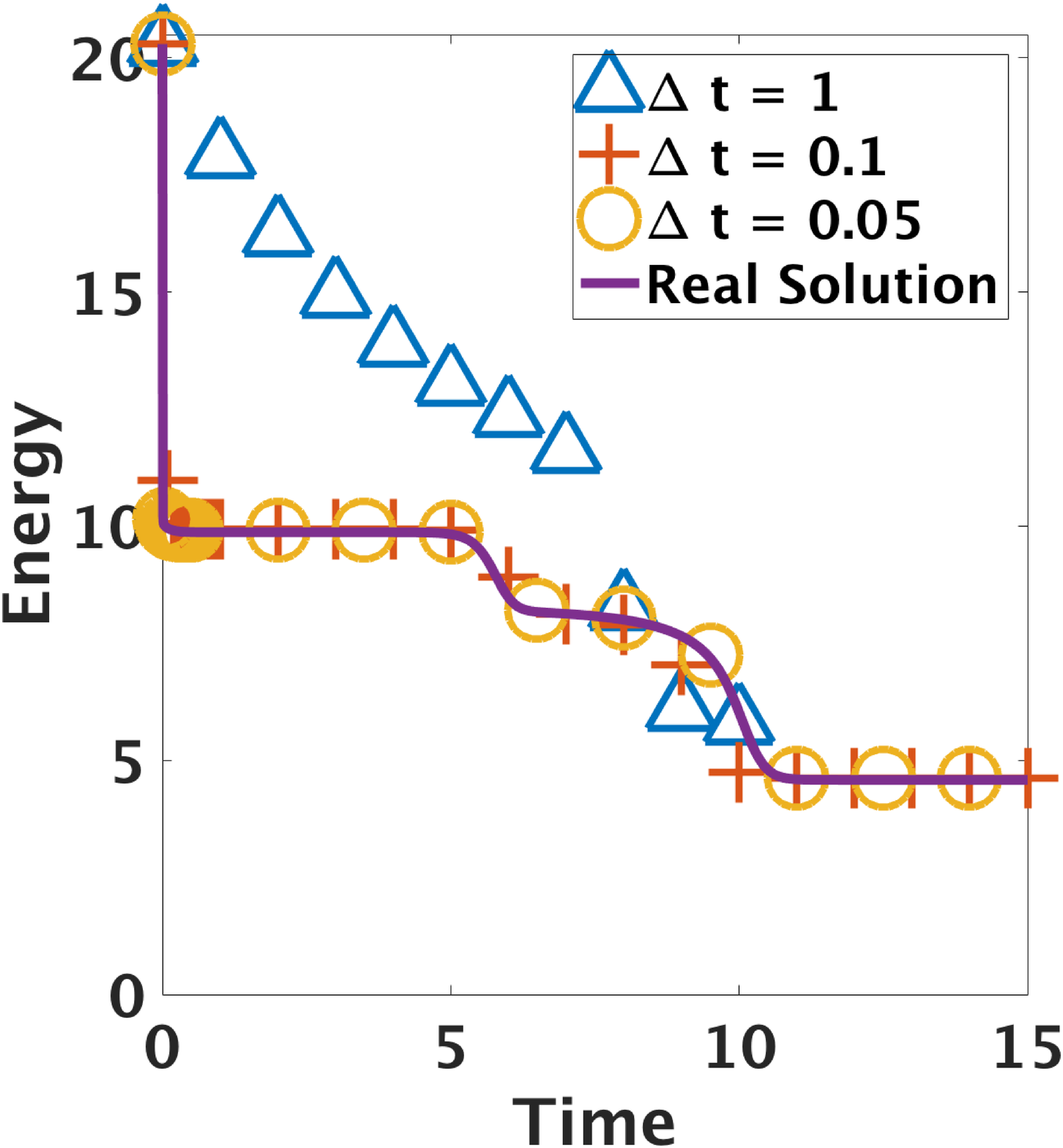}}
\caption{Energy evolution calculated by different SAV schemes with various time step sizes. This figure demonstrates the HSAV scheme could predict accurate energy dissipation dynamics with much larger time steps than the SAV scheme while solving the MBE model with slope selection.}
\label{fig:MBE-energy}
\end{figure}

Besides, the total CPU times using each scheme to calculate the MBE model till $t=15$ is summarized in Table  \ref{tab:MBE-CPU}, where we observe the HSAV scheme is much faster than the SAV scheme, as much larger time steps can be used for HSAV scheme while preserving the desired accuracy.

\begin{table}[H]
\caption{ Total CPU time using various numerical schemes solving the MBE model.}
\label{tab:MBE-CPU}
\scriptsize
\centering
\begin{tabular}{c|c|c|c}
\hline
 & SAV-CN Scheme &  HSAV 4th-order Scheme  & HSAV 6th-order Scheme \\
\hline
$\delta t$ &  $0.0001$ & $0.025$ & $0.05$ \\
\hline
CPU time (seconds) &  672.39 & 79.45  & 74.72 \\
\hline
\end{tabular}
\end{table}

One simulation using fourth order HSAV scheme with the time step $\delta t=0.025$ is shown in Figure \ref{fig:MBE-coarsening}, where the height profile of $\phi$ at different times are shown. The patterns agree very well with other numerical solvers in literature, while we could use an extremely larger time step than the time-step used in other literature.
\begin{figure}[H]
\center
\subfigure[$t=0$]{\includegraphics[width=0.25\textwidth]{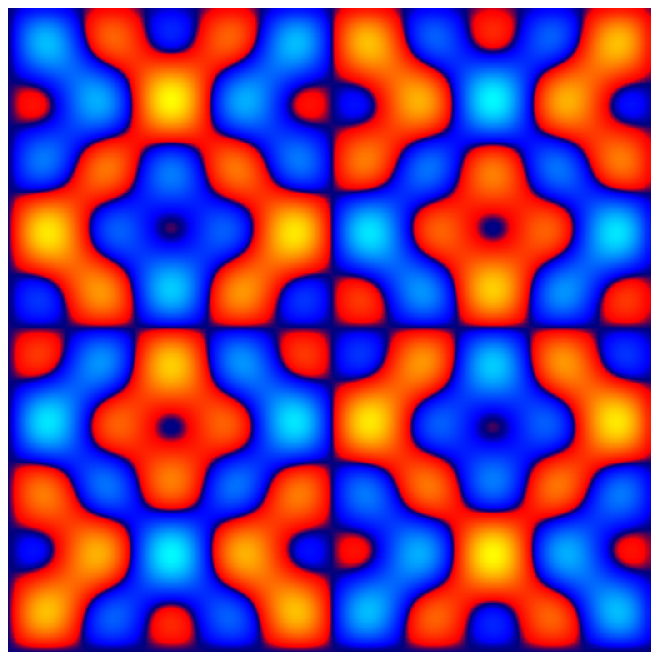}}
\subfigure[$t=0.025$]{\includegraphics[width=0.25\textwidth]{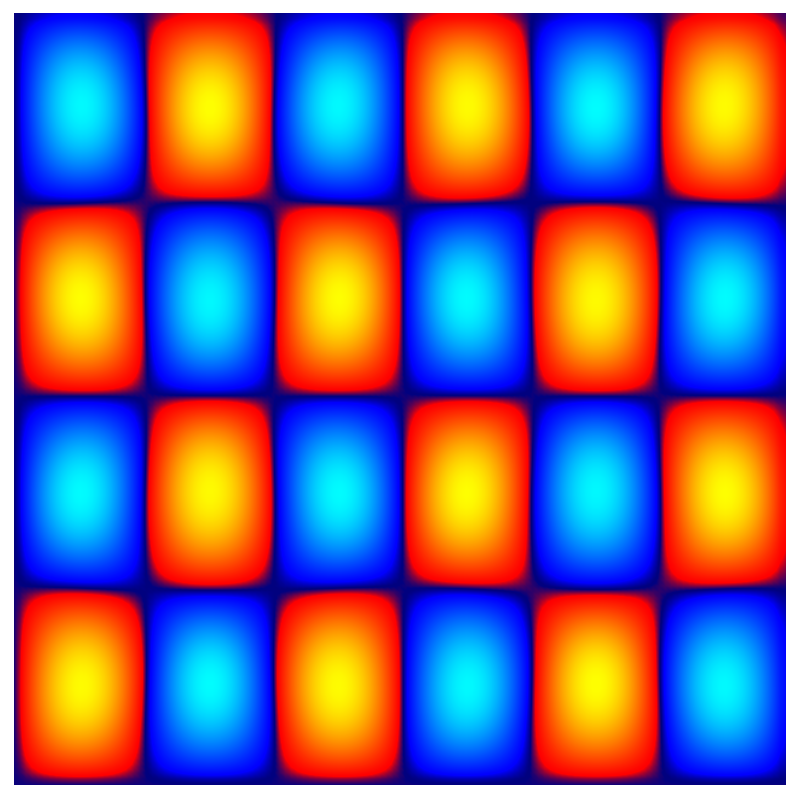}}
\subfigure[$t=2.5$]{\includegraphics[width=0.25\textwidth]{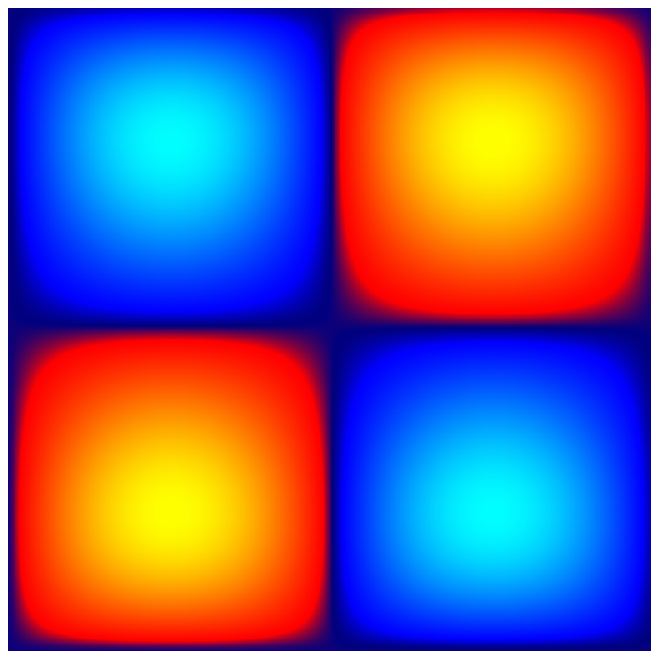}}

\subfigure[$t=5.5$]{\includegraphics[width=0.25\textwidth]{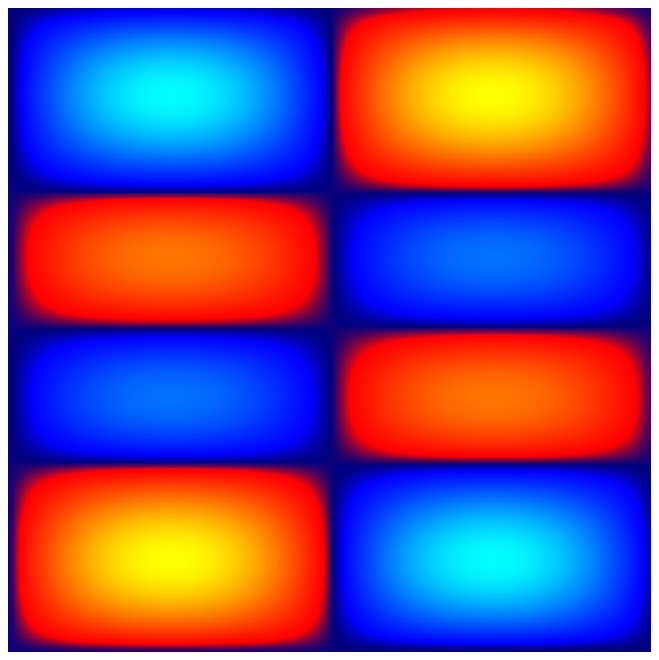}}
\subfigure[$t=8$]{\includegraphics[width=0.25\textwidth]{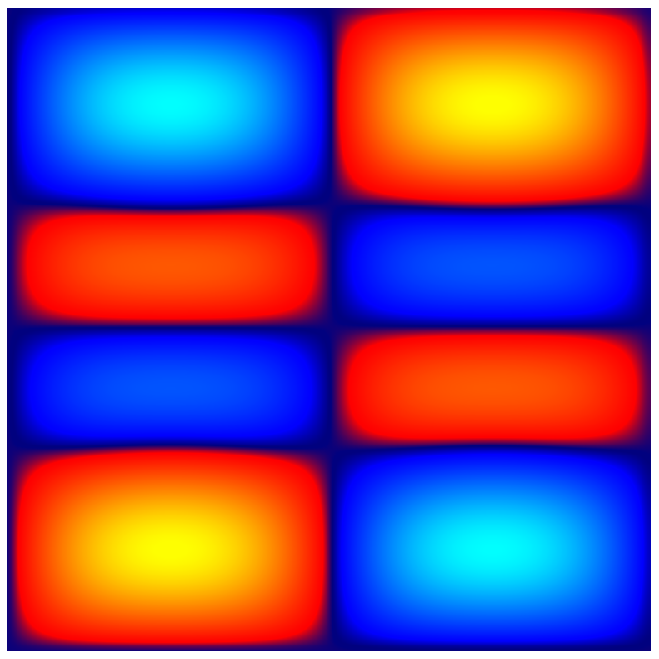}}
\subfigure[$t=30$]{\includegraphics[width=0.25\textwidth]{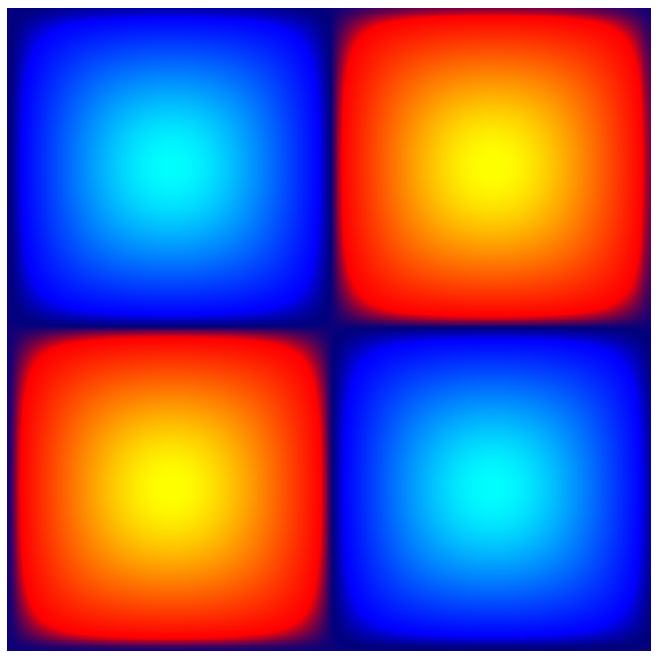}}
\caption{The isolines of numerical solutions of the height function $\phi$ for the MBE model with slope selection using the 4th order HSAV  scheme. The time step is $\delta t=0.025$. Snapshots are taken at $t=0, 0.05, 2.5, 5.5, 8, 30$, respectively.}
\label{fig:MBE-coarsening}
\end{figure}

\section{Conclusion}
In this paper, we combine the SAV approach with the structure-preserving discretization to propose a new class of energy stable methods for gradient flow models, which we name it the HSAV scheme. The proposed HSAV scheme could reach arbitrarily high-order accuracy in time while respecting the discrete energy dissipation law in term of the modified free energy of the SAV equivalent system. Therefore, the proposed schemes can be used to conduct longtime dynamic simulations for gradient flow problems with larger time steps. Some numerical benchmarks are presented to illustrate the excellent performance of the proposed numerical methods. Note that the proposed HSAV method is rather general to be applied for any gradient flow models derived through energy variation. Furthermore, it could also be generalized to study thermodynamically-consistent hydrodynamic models, which will be pursued in our later research.

\section*{Acknowledgment}
Yuezheng Gong's work is partially supported by the Natural Science Foundation of Jiangsu Province (Grant No. BK20180413) and the National Natural Science Foundation of China (Grant No. 11801269).  Jia Zhao's work is partially supported by National Science Foundation under grant number NSF DMS-1816783.

\bibliographystyle{plain}

\end{document}